\documentclass[]{interact}

\usepackage{epstopdf}
\usepackage[caption=false]{subfig}

\usepackage[numbers,sort&compress]{natbib}
\bibpunct[, ]{[}{]}{,}{n}{,}{,}
\makeatletter
\def\NAT@def@citea{\def\@citea{\NAT@separator}}
\makeatother

\theoremstyle{plain}
\newtheorem{theorem}{Theorem}[section]
\newtheorem{lemma}[theorem]{Lemma}
\newtheorem{corollary}[theorem]{Corollary}
\newtheorem{proposition}[theorem]{Proposition}

\theoremstyle{definition}
\newtheorem{definition}[theorem]{Definition}
\newtheorem{example}[theorem]{Example}

\theoremstyle{remark}
\newtheorem{remark}{Remark}

\usepackage [latin1]{inputenc}
\usepackage{mathabx}

\newcommand{\pprec}{{\prec\!\!\prec}}
\newcommand{\ssucc}{{\succ\!\!\succ}}
\newcommand{\eq}{\mathbin{<\!>}}

\begin{document}


\title{Internal structure and analytical representation of compatible preference relations defined on infinite-dimensional real vector spaces.}

\author{\name{Valentin~V. Gorokhovik\thanks{Email: gorokh@im.bas-net.by}}\thanks{ID: https://orcid.org/0000-0003-2447-5943}
\affil{Institute of Mathematics,
The National Academy of Sciences of Belarus, \\ Minsk, Belarus}
}

\maketitle

\begin{abstract}
The paper deals with partial and weak preference relations defined on infinite-dimensional vector spaces and compatible with algebraic operations. By a partial preference we mean an asymmetric and transitive binary relation, while a weak preference is such a partial preference for which the indifference relation corresponding it is transitive (an indifference relation is the complement to the union of a partial preference and the inverse to it). Our first aim is to study the internal structure of compatible partial and weak preferences. We suppose that the internal structure of a compatible partial preference is elementary if its positive cone is relatively algebraic open, and we refer to such compatible partial preferences as relatively open ones. It is proved that an arbitrary compatible partial preference is the disjoint union of the partially ordered family of relatively open compatible partial preferences and, moreover, as an ordered set this family is an upper lattice. We identify the structure of this upper lattice with the internal structure of a compatible partial preference. The internal structure of compatible weak preferences is more refined: each compatible weak preference is the disjoint union of the chain (the linear ordering family) of relatively open compatible partial preferences the restriction of each of which to the linear hull of its positive cone is a weak preference. Using the chain characterizing the internal structure of a compatible weak preference we define a step-linear function that analytically represents this compatible weak preference. Further, we prove that each compatible partial preference can be regularly extended to a compatible weak preference and, moreover, each compatible partial preference is the intersection of its regular extensions to a compatible weak preference. Due to the latter assertion we obtain an analytical representation of a compatible partial preference by the family of step-linear functions.
\end{abstract}



\section{Introduction}

Preference relations are central objects in optimization theory (in
particular, in vector optimization), decision making, mathematical economics, social sciences and others mathematical theories and their applications. A preference relation is generated by some ordering (it may be a preorder, a partial order, a strict partial order) which is defined on the set of feasible alternatives. In those cases when the set of feasible alternatives is a subset of some vector space it is usually assumed that the ordering which generates a preference relation are compatible with algebraic operations of the vector space (the exact definition of a compatible binary relation is given in Section 2). A compatible ordering is uniquely determined by the cone of positive vectors corresponding it, and therefore properties of a compatible ordering are completely characterized through those of its positive cone.

In most works devoted to vector optimization problems, the authors usually assume that the compatible ordering of the image space of the objective vector-valued function is such that its cone of positive elements has a nonempty interior in some topological or algebraic sense. Such notions as an algebraic interior (also called a core), a relative algebraic interior (also called an intrinsic core), a quasi-interior, a quasi-relative interior and others (see, e.g., \cite{Adan_Novo1,Adan_Novo2,Bao,Cuong1,Cuong2,Cuong3,Grad,Gunter,Hernandez,Khazayel,NZ,Millan,Zalinescu} and the references therein) are used as  generalizations of the conventional topological interiority notion.
The assumptions on nonemptiness interiority of the positive cone are caused not by the essence of vector optimization problems themselves, but by the methods of their study, in particular, by those which are based on theorems (classical or generalized ones) on the separation of convex sets by hyperplanes (or, equivalently, by linear functionals). It is clear, that these methods are not applicable to general vector optimization problems for which a priory assumptions on nonempty interiority of a positive cone do not hold.
A method for studying general (without any a priori assumptions on the positive cone) convex vector optimization problems was first proposed by Martinez-Legaz in \cite{ML88} in the finite-dimensional setting and by Gorokhovik in \cite{Gor86,Gor-m} in the infinite-dimensional setting. Later the results from \cite{Gor86,Gor-m} were developed in \cite{GS97,GS98a,GS99,GS2000,Gor21}. This method is based on the separation of convex sets by halfspaces \cite{Kakutani,Tukey,HF} and on the analytical representations of halfspaces developed in \cite{ML_S87,ML_S88} in the finite-dimensional setting and in \cite{Gor-m,GS97,GS98a,GS99,GS2000,Gor21} in the infinite-dimensional setting.

Using tools developed in \cite{ML_S87,ML_S88}, Martinez-Legaz and Singer in \cite{M-L1} have comprehensively studied compatible preorders defined on finite-dimensional vector spaces.

The purpose of this paper is to study compatible ordering relations defined on infinite-dimensional vector spaces. As a basic object for studying we choose a strict partial order, which is (see, for instance, \cite{Harzheim}) an asymmetric and transitive binary relation and to which we refer here as a partial preference relation (or, shortly, a partial preference). Our choice is quite general, since the asymmetric part of any preorder and any partial order is asymmetric and transitive, so the results obtained here for partial preference relations can be easily extended to preorders and partial orders.

The paper is organized as follows.

Section 2 contains preliminaries on general and compatible partial preferences (more details can be found in \cite{Harzheim,Gor-m21,Pere,Jameson,Schaefer,Wong_Ng,Arrow,Fishburn,Craven}).

In Section 3 we study the internal structure of compatible partial preferences. We suppose that the internal structure of compatible partial preferences, whose positive cone is relatively algebraically open, is elementary. We refer to such compatible partial preferences as relatively open ones. Our main result of Section 3 is the following: each compatible partial preference is the disjoint union of the partially ordered family of relatively open compatible partial preferences and, moreover, this family, as an ordered set, is an upper lattice.  We identify the structure of this upper lattice with the internal structure of the corresponding compatible partial preference.

Section 4 is devoted to compatible weak preferences. A weak preference is such a partial preference for which the indifference relation corresponding it is transitive (an indifference relation is the complement to the union of a partial preference and the inverse to it). We show that the internal structure of compatible weak preferences is more refined: each compatible weak preference is the disjoint union of the chain (the linear ordering family) of relatively open compatible partial preferences the restriction of each of which to the linear hull of its positive cone is a weak preference. It also is shown that for each relatively open compatible partial preference which belongs to the chain characterizing the internal structure of a compatible weak preference there exists a nonzero linear function which is strict monotone with respect to it.
Using this fact the internal structure of a compatible weak preference is associated with a chain of nonzero linear functions which then is used in the next section for constructing a real-valued function analytically representing the compatible weak preference.

Section 5 is started with recalling the definition of a step-linear function. Step-linear functions were firstly introduced in \cite{Gor-m} and then they were more thorough studied in \cite{GS97,GS98a,GS99,GS2000,Gor21}.
The main result of Section 5 is Theorem 5.5 in which we prove that each compatible weak preference admits an analytical representation by a step-linear function.

In Section 6 we prove that each compatible partial preference can be regularly extended to a compatible weak preference and, moreover, each compatible partial preference is the intersection of its regular extensions to a compatible weak preference. The first of these assertions is a counterpart of the Szpilrajn theorem \cite{Szpilrajn}, while the second is a counterpart of the Dushnik--Miller theorem \cite{Dushnik} from the general theory of ordered sets.
Due to these counterparts of classical theorems we obtain an analytical representation of a compatible partial preference by the family of step-linear functions.

\section{Preliminaries.}

The main purpose of this section is to recall the necessary concepts related to binary relations and, in particular, to ordering relations, and also to specify the terminology adopted in this paper. More details on orderings and preference relations can be found, for example, in \cite{Harzheim,Pere,Jameson,Schaefer,Wong_Ng,Arrow,Fishburn,Craven,Gor-m21}.

\subsection{Preference relations}

Let $Y$ be a set and let $\prec \, \subset Y \times Y$ be a binary relation on $Y.$ In applications, in particular, in economics and in the decision making theory, $Y$ is interpreted as the set of feasible commodity bundles or the set of feasible alternatives and a binary relation $\prec$ defined on $Y$ as a preference relation of a consumer or a decision maker. Within the framework of such interpretation the assumption that $\prec$ is \textit{asymmetric} ($y_1 \prec y_2 \Longrightarrow y_2  \hspace{5pt}/\hspace{-10pt}\prec y_1$, where $/\hspace{-10pt}\prec:= (Y\times Y)\,\,\setminus \prec$ is the negation of $\prec$) seems to be quite natural. Indeed, if an alternative $y_1$ is 'better' than an alternative $y_2$ then $y_2$ can't be 'better' than $y_1.$

Each asymmetric binary relation $\prec \, \subset Y \times Y$ generates two other binary relations on $Y$:
the \textit{indifference relation} $\asymp$ defined by
$$
y_1 \asymp y_2 \Longleftrightarrow y_1  \hspace{5pt}/\hspace{-10pt}\prec y_2\,\,\text{and}\,\,y_2 \hspace{5pt}/\hspace{-10pt}\prec y_1;
$$
and the \textit{equipotency relation} $\sim$ defined by
$$
y_1 \sim y_2 \Longleftrightarrow  \{y \in Y \mid y \asymp y_1\} = \{y \in Y \mid y \asymp y_2\}.
$$

It is worth notice that the equipotency relation $\sim$ is not defined directly through the asymmetric
relation $\prec$, but through the indifference relation $\asymp$ generated by $\prec$.

The indifference relation $\asymp$ is reflexive and symmetric
(such relations are called tolerance ones), while the equipotency relation $\sim$ is reflexive, symmetric, and transitive, and, consequently, $\sim$ is an equivalence relation. In general, the inclusion $\sim \subset \asymp$ is true, while the equality $\asymp = \sim$ holds if and only if the indifference relation $\asymp$ is transitive.

Along with asymmetry, the property of transitivity ($y_1 \prec y_2, y_2 \prec y_3 \Longrightarrow y_1 \prec y_3$) is often (but not always, see \cite{Fishburn91,Rub_Gas,Gor_Tr16}) imposed on a preference relation $\prec.$ From now on, any asymmetric and transitive binary relation defined on $Y$ will be referred to as \textit{a partial preference relation}, in short, \textit{a partial preference} on $Y$.

The equipotency relation $\sim$ generated by a partial preference $\prec$ can be defined directly through $\prec$ as follows:
\begin{equation}\label{e2.1a}
y_1 \sim y_2 \Longleftrightarrow \left\{\!\! \begin{array}{ll}
\{y \in Y \mid y \prec y_1\}=\{y \in Y \mid y \prec y_2\}, \\
\{y \in Y \mid y_1 \prec y\}=\{y \in Y \mid y_2 \prec y\}.\\
\end{array}
\right.
\end{equation}

Furthermore, for a partial preference $\prec$ and the equipotency relation $\sim$ generated by it the following implications
$$y \prec z, z \sim u \Longrightarrow y \prec u\,\,\text{\rm and}\,\,y \sim z, z \prec u \Longrightarrow y \prec u$$ hold for any $y,z,u \in Y$.

The union of a partial preference $\prec$ with the equipotency relation $\sim$ generated by $\prec$, denoted by $\precsim := \prec \bigcup \sim$, is a preorder relation (a reflexive and transitive one) on $Y$, and $\sim$ is the greatest (by inclusion) equivalence relation on $Y$ such that the union of it with a partial preference $\prec$ is a preorder relation.

Note that the union of a partial preference relation $\prec$ with the equality (identity) relation $\Delta:=\{(y,y) \in Y \times Y \mid y \in Y\}$ is a reflexive, transitive, and antisymmetric binary relation, i.e., $\prec\bigcup \Delta$ is a partial order on $Y.$ Conversely, if $\preceq$ is a partial order on $Y$ then its asymmetric part $\prec := \preceq \setminus \Delta$  is a partial preference relation (also called a strict partial order). So, partial orders defined on $Y$ correspond one-to-one with partial preferences on $Y$.

A partial preference $\prec \, \subset Y \times Y$ is called \textit{a weak preference} if the indifference relation $\asymp$ generated by it is transitive or, equivalently, if $\asymp = \sim$.

An asymmetric relation $\prec$ is a weak preference if and only if it is negatively transitive, i.e., if $/\hspace{-10pt}\prec$, the negation of $\prec$,   is transitive.

Furthermore, a partial preference $\prec$ is a weak preference  if and only if the preorder relation $\precsim := \prec \bigcup \sim$ is \textit{total}, i.e., if  and only if for any $y_1,y_2 \in Y$ at least one (or even both) of the inequalities, either $y_1 \precsim y_2$ or $y_2 \precsim y_1$, holds.

A partial preference $\prec \, \subset Y \times Y$ is called \textit{a perfect (linear) preference} if for any $y_1,y_2 \in Y$ such that $y_1 \ne y_2$ either $y_1 \prec y_2$ or $y_2 \prec y_1$ holds.

\subsection{Compatible preference relations on real vector spaces}\label{sec-2.2}

When $Y$ is a real vector space, a binary relation (not necessarily a partial preference one) $\prec \, \subset Y \times Y$ is said to be \textit{compatible} with the algebraic operations defined on $Y$ (see, e.g., \cite{Pere,Jameson,Schaefer,Wong_Ng}) if for any $y_1,y_2,y \in Y$ and a real $\lambda > 0$ the following two implications hold:
\begin{equation}\label{e1}
y_1 \prec y_2 \Longrightarrow \lambda y_1 \prec \lambda y_2
\end{equation}
\begin{equation}\label{e2}
y_1 \prec y_2 \Longrightarrow y_1 + y \prec y_2 + y.
\end{equation}

If a compatible binary relation $\prec$ is transitive then for any $y_1,y_2,z_1,z_2 \in Y$ such that $y_1 \prec y_2$ and $z_1 \prec z_2$ one has $y_1 + z_1 \prec y_2 + z_2$. Indeed, from $y_1 \prec y_2$, $z_1 \prec z_2$, and \eqref{e2} we obtain $y_1 +z_1 \prec y_2 + z_1$, $y_2 + z_1  \prec y_2 + z_2$. Further, through the transitivity of $\prec$ we come to $y_1 + z_1 \prec y_2 + z_2$.

It follows from the implications \eqref{e1} and \eqref{e2} that the set $P_\prec := \{y \in Y \mid 0 \prec y \}$ associated with a compatible binary relation $\prec$ is a cone in $Y$ and, moreover,
\begin{equation}\label{e3}
y_1 \prec y_2 \Longleftrightarrow y_2 - y_1 \in P_\prec.
\end{equation}

The cone $P_\prec$ is called \textit{the cone of positive elements} or  \textit{the positive cone} of the compatible relation $\prec$. 

Recall that a set $K \subset Y$ is called a cone if $\lambda y \in K$ for all $y \in K$ and all $\lambda >0$. A cone $K$ is convex if and only if $K+K \subset K$.  A convex cone $K$ is said to be \textit{asymmetric} if $K\cap(-K)= \varnothing$.

A compatible binary relation $\prec$ defined on a real vector space $Y$ is asymmetric and transitive (i.e., $\prec$ is a partial preference relation) if and only if the cone of positive elements $P_\prec$ associated with $\prec$ is asymmetric ($P_\prec \bigcap (-P_\prec) = \varnothing$) and convex ($P_\prec + P_\prec \subset P_\prec$).

Thus, there is a one-to-one correspondence between the collection of partial preference relations which are
compatible with the algebraic operations of a vector space $Y$ and the collection of asymmetric and convex cones from $Y.$

Taking into account the one-to-one correspondence between partial orders and partial preference relations, a real vector space $Y$ endowed with a preference relation $\prec$ that is compatible with the algebraic operation of $Y$ will be called an ordered vector space (cf. with \cite{Pere,Jameson,Schaefer,Wong_Ng}) and will be denoted by $(Y,\prec).$

Let $K$ be a convex cone in $Y.$ The set
$$
L_K:=\{h \in Y\,|\,y+th \in K\,\,\text{for all}\,\,y \in
K\,\,\text{and all}\,\, t \in {\mathbb{R}}\}
$$
is the greatest vector subspace in $Y$ for which the equality $L_K + K = K$ holds.

We refer to $L_K$ as {\it the vector space associated with the convex cone} $K.$

It is easy to verify that $K\bigcup L_K$ also is a
convex cone. A convex cone $K$ is asymmetric if and only if  $K \bigcap
L_K = \varnothing.$ Clearly $L_K = L_{-K}$ for any convex cone $K \subset Y$. Furthermore, for any convex cone $K \subset Y$ the inclusion $L_{K} \subset {\rm Lin}(K)$, where ${\rm Lin}(K):= K - K$ is the linear hull of $K$, is true. Indeed, for every $y \in K$ we have $y+L_K \subset K$ and, consequently, $L_K = (y + L_K) - y \subset K-K = {\rm Lin}(K)$.

The indifference relation $\asymp$ and the equipotency relation $\sim$ generated by a compatible asymmetric relation, in particular, by a compatible partial preference $\prec$, also are compatible with algebraic operations of $Y$. Besides, for the equipotency relation $\sim$ generated by a compatible partial preference $\prec$ we have
\begin{equation}\label{e3a}
y_1 \sim y_2 \Longleftrightarrow y_2 - y_1 \in L_\prec.
\end{equation}
where $L_\prec := L_{P_\prec}$ stands for the vector space associated with the positive cone $P_\prec$.

To prove \eqref{e3a} we note that for a compatible partial preference $\prec$ the definition of the equipotency relation $\sim$ generated by $\prec$ (see \eqref{e2.1a}) can be rewritten as follows:
\begin{equation*}
y_1 \sim y_2 \Longleftrightarrow y_1 + P_\prec = y_2 + P_\prec\,\,\text{and}\,\, y_1 - P_\prec = y_2 - P_\prec.
\end{equation*}
We get from $y_1 + P_\prec = y_2 + P_\prec$ that $y_1 - y_2 + P_\prec = P_\prec$ and, since $P_\prec$ is a cone, that $\alpha(y_1 - y_2) + P_\prec = P_\prec\,\,\forall\,\,\alpha > 0$. Similarly, from the equality $y_1 - P_\prec = y_2 - P_\prec$ we get $\alpha(y_1 - y_2) + P_\prec = P_\prec\,\,\forall\,\,\alpha < 0$. Thus, $y_1 \sim y_2 \Longrightarrow \alpha(y_1 - y_2) + P_\prec = P_\prec\,\,\forall\,\,\alpha \in {\mathbb{R}} \Longrightarrow y_1 - y_2 \in L_\prec$.

Conversely, let $y_1 - y_2 \in L_\prec$. By the definition of $L_\prec$ we have $y_1 - y_2 + P_\prec \subset P_\prec$ or, equivalently, $y_1 + P_\prec \subset y_2 + P_\prec$. But, since $y_2-y_1 \in L_\prec$, we have also $y_2 + P_\prec \subset y_1 + P_\prec$ and, consequently, $y_1 + P_\prec = y_2 + P_\prec$. Using the equality $L_\prec = L_{P_\prec} = L_{(-P_\prec)}$ we get in a similar way that $y_1 - P_\prec = y_2 - P_\prec$. This completes the proof of \eqref{e3a}.

A subset $H$ of a real vector space $Y$ is called \textit{a halfspace} if both $H$ and its complement $Y\setminus H$ are convex. A halfspace $H$ which is a cone is called \textit{a conical halfspace}. An asymmetric convex cone $H \subset Y$ is a conical halfspace if and only if the set $Y \setminus (H\bigcup(-H))$ is a vector subspace in $Y$. Moreover, for an asymmetric conical halfspace $H \subset Y$ the equality $L_H = Y \setminus (H\bigcup(-H))$ holds.

It is evident that for any halfspace $H$ its complement $Y  \setminus H$ is a halfspace too. A widespread class of pairs of complementary halfspaces is given by the subsets
$$
H_>(\phi, \gamma):=\{y \in Y \mid \phi(y) > \gamma\}\,\,\text{and}\,\,H_\le(\phi, \gamma):=\{y \in Y \mid \phi(y) \le \gamma\},
$$
where $\phi:Y \to {\mathbb{R}}$ is non-zero linear function on $Y$ and $\gamma \in {\mathbb{R}}$ is an arbitrary real.

Semispaces and their complements give other examples of complementary halfspaces. Recall \cite{Hammer,Klee} that a subset $S \subset Y$ is called \textit{a semispace generated by a point $a \in Y$} (or, simply, \textit{a semispace at a point $a \in Y$}) if it is a maximal (by inclusion) convex set in $Y \setminus \{a\}$.

Semispaces were introduced and studied by Hammer \cite{Hammer} and Klee \cite{Klee}.  Halfspaces (also called hemispaces) in finite-dimensional spaces were studied by Lassak \cite{Las} and Martinez-Legaz and Singer \cite{ML_S88} and in infinite-dimensional spaces by Lassak and Pruszynski \cite{LasPro}, Gorokhovik and Semenkova \cite{GS98}, and Gorokhovik and Shinkevich \cite{GS99,GS2000}. It is worth noting that in the English translation of the paper \cite{GS98}, halfspaces are mistakenly translated as semispaces.

A compatible partial preference $\prec$ is a compatible weak preference if and only if its positive cone $P_\prec$ is a asymmetric conical halfspace, while a compatible partial preference $\prec$ is a compatible perfect (linear) preference if and only if its positive cone $P_\prec$ is a semispace at the origin.

\section{The internal structure of compatible partial preferences.}

Let $(Y,\prec)$ be an ordered vector space.

A vector $z \in Y$ is called \textit{strongly positive}, this is denoted by $0 \pprec z$ or $z \ssucc 0$, if for any $y \in P_\prec$ there is a real number $\mu >0$ such that $\mu y \prec z$.

The set of all strongly positive vectors of $Y$ is denoted by $P_\pprec$. It follows immediately from the definition that $P_\pprec \subset P_\prec$. The next example shows that $P_\pprec$ is generally a proper subset of $P_\prec$.

\begin{example}
{\rm Let $Y = {\mathbb{R}}^2$ and $P_\prec = {\mathbb{R}}^2_+ \setminus \{0\}$. It is easily seen that the vectors $(\alpha,0)$ and $(0,\alpha)$ with $\alpha > 0$ are positive, but not strongly positive, while any vectors $(\alpha_1,\alpha_2)$ with $\alpha_1 > 0, \alpha_2 > 0$ are strongly positive.}
\end{example}

To characterize strongly positive vectors we need the following notions (see \cite{Khazayel,Klee51,Raikov,Rubinstein1,Rubinstein,Holmes,Aliprantis,KhTZ,Millan}).

\textit{The algebraic interior} (or \textit{the core}) of a subset $Q \subset Y$, denoted by ${\rm cr}Q$,  is the subset of $Q$ consisting of all points $z \in Q$ such that the intersection of $Q$ with each straight line of $Y$ passing through $z$ contains an open interval around $z$.

\textit{The relative algebraic interior} (or \textit{the intrinsic core}) of a subset $Q \subset Y$, denoted by ${\rm icr}Q$, is the subset of $Q$ consisting of all points $z \in Q$ such that the intersection of $Q$ with each straight line laying in the affine hull of $Q$ (${\rm aff}Q$) and passing through $y$ contains an open interval around $z$.

More precisely, $z \in {\rm cr}Q$ if and only if for any $y \in Y \setminus \{z\}$  there exists a positive real number $\delta > 0$ such that $z+t(y-z) \in Q$ for all $t \in (-\delta,\delta)$, while $z \in {\rm icr}Q$ if and only if for any $y \in {\rm aff}Q \setminus \{z\}$  there exists a positive real number $\delta > 0$ such that $z+t(y-z) \in Q$ for all $t \in (-\delta,\delta)$.

If a set $Q \subset Y$ is convex then $z \in {\rm icr}Q$ if and only if for any $y \in Q \setminus \{z\}$  there exists a positive real number $\delta > 0$ such that $z+t(y-z) \in Q$ for all $t \in (-\delta,\delta)$

 A set $Q \subset Y$ is called \textit{(relatively) algebraic open} if it coincides with its (relative) algebraic interior, i.e., if $Q = {\rm cr}Q$ ($Q = {\rm icr}Q$).

\begin{theorem}\label{th3.2}
Let $(Y, \prec)$ be an ordered vector space. A vector $z \in Y$ is strongly positive if and only if it belongs to the relative algebraic interior of $P_\prec$, i.e., $P_\pprec = {\rm icr}P_\prec$.
\end{theorem}

\begin{proof}
Suppose that $P_\pprec \ne \varnothing$ and choose an arbitrary vector $z \in P_\pprec$. By the definition of $P_\pprec$ for any $y \in P_\prec$ there exists a real $\mu > 0$ such that $\mu y \prec z$. Since $ty \prec \mu y$ for all $t \in (0,\mu)$, then through transitivity of $\prec$ we have $ty \prec z$ for all $t \in (0,\mu)$ or, equivalently, $z - ty \in P_\prec$ for all $t \in (0,\mu)$. Besides, it follows from the inclusion $P_\pprec \subset P_\prec$ that $tz \in P_\prec$ for all $t > 0$. Consequently, using convexity of the positive cone $P_\prec$ we get $z -t(y-z) \in P_\prec$ for all $t \in (0, \mu)$. Moreover, since both $z$ and $y$ belongs to the convex cone $P_\prec$ then $z + t(y-z) \in P_\prec$ for all $t \in [0,1]$. Hence, $z +t(y-z) \in P_\prec$ for all $t \in (-\delta,\delta)$, where $\delta = \min\{\mu, 1\}$. Due to the arbitrary choice of the vector $y \in P_\prec$ and convexity of $P_\prec$, we conclude that $z \in {\rm icr}P_\prec$. It proves the inclusion $P_\pprec \subset {\rm icr}P_\prec$.

To prove the converse inclusion we suppose that ${\rm icr}P_\prec \ne \varnothing$. Let $z \in {\rm icr}P_\prec$. Then for any $y \in P_\prec$ there is a real $\delta > 0$ such that $z + t(y-z) \in P_\prec$ for all $t \in (-\delta,\delta)$. This implies that $z -\displaystyle\frac{|t|}{1+|t|}y \in P_\prec$ for all $t \in (-\delta,0)$. Since $\displaystyle\frac{|t|}{1+|t|} > 0$ and the choice of $y \in P_\prec$ was arbitrary we conclude that $z \in P_\pprec$. Thus, ${\rm icr}P_\prec \subset P_\pprec$ and, consequently, $P_\pprec = {\rm icr}P_\prec$.
\end{proof}

It follows from Theorem \ref{th3.2} that strongly positive vectors exist in any finite-dimensional ordered vector space. At the same time the  next example shows that on each infinite-dimensional vector space there exist such compatible partial preferences $\prec$ for which $P_\pprec = \varnothing$.

\begin{example}\cite[Ch.2, $\S$ 7]{Raikov} 
Let $Y$ be a real vector space and let $\{e_i, i \in I\}$ be a Hamel basis for $Y$. Let us define the compatible partial preference $\prec$ on $Y$ the positive cone of which is the convex cone consisting of non-zero vectors $y \in Y$ whose components $\{y_i,i \in I\}$ in the given basis $\{e_i, i \in I \}$ is nonnegative. Since the basis $\{e_i, i \in I\}$ is infinite, each vector $y \in P_\prec$ has zero components. Suppose that $y_j =0$. Since $e_j \in P_\prec$ and for any real $\mu > 0$ the vector $y - \mu e_j$ does not belong to $P_\prec$, then $y \notin {\rm icr}P_\prec$ and, hence, ${\rm icr}P_\prec = \varnothing$ as $y$ is an arbitrary vector from $P_\prec$.
\end{example}

\begin{remark}
Strong positive vectors were firstly introduced by Krein and Rutman in \cite{Krein} for ordered normed spaces as vectors which belong to ${\rm int}P_\prec$. Since ${\rm cr}P_\prec = {\rm int}P_\prec$, provided that ${\rm int}P_\prec \ne \varnothing$, the concept of strongly positive vectors introduced above extends the definition of Krein and Rutman to arbitrary ordered vector spaces.

Furthermore, when the positive cone $P_\prec$ is generating, i.e., when $P_\prec - P_\prec = Y$, the notion of a strongly positive vector coincides with that of an order unit. Recall \cite{Schaefer,Pere,Jameson}, that a vector $e$ of an ordered vector space $Y$ is called an order unit if the order interval $[-e,e] := \{y \in Y \mid -e \prec y \prec e\}$ is absorbing in $Y$, this means that for any $y \in Y$ there is a real $\lambda >0$ such that $-\lambda e \prec y \prec \lambda e$.
\end{remark}

The strong positivity of a vector $z \in P_\prec$ means that the vector $z \in P_\prec$ majorizes (in the certain sense) all other positive vectors $y \in P_\prec$. Restricting this property to pairs of positive vectors, we get the following definition of the majorization relation on the positive cone.

\begin{definition}
    We say that a positive vector $z \in P_\prec$ \textit{majorizes} a positive vector $y \in P_\prec$, and denote this by $y \unlhd z$, if there is a positive real $\mu > 0$ such that $\mu y \prec z$.
\end{definition}

It is not difficult to verify that $\unlhd$ is a preorder relation (i.e., a reflexive and transitive binary relation) on $P_\prec$. By the symbol $\eq$ we denote the symmetric part of $\unlhd$ defined by $y_1 \eq y_2 \Leftrightarrow y_1 \unlhd y_2, y_2 \unlhd y_1$. It is easy to see that $y_1 \eq y_2$ if and only if there exist reals $\mu > 0$ and $\nu > 0$ such that $\mu y_1 \prec y_2 \prec \nu y_1$. The asymmetric part of $\unlhd$ denoted by $\lhd$ is defined by $y_1 \lhd y_2 \Leftrightarrow y_1 \unlhd y_2, y_2 \hspace{5pt}/\hspace{-10pt}\unlhd y_1$. The relation $\lhd$ is a strict partial order, while  $\eq$ is an equivalency relation.

\begin{remark}
The majorization relation $\unlhd$ on the positive cone of a compatible partial preference relation $\prec$ is related to the relation introduced by Rubinstein in \cite[p.~187]{Rubinstein}  for studying the face structure of convex sets (see also \cite{Rubinstein1,Klee}).
\end{remark}

It immediately follows from the properties of a compatible partial preference $\prec$ and the definition of the majorization relation $\unlhd$ that for any $y,z \in P_\prec$ the implications
\begin{equation}\label{e3.4}
y \prec z \Longrightarrow y \unlhd z\,\,\,\,\text{and}\,\,\,\,y \sim z \Longrightarrow y \eq z
\end{equation}
hold.

Moreover, for any $y,y_1,y_2,z \in Y$ and any positive reals $\alpha > 0, \beta > 0$ the following implications also are true:
\begin{equation}\label{e3.5}
y \unlhd z \Longrightarrow \alpha y \unlhd \beta z;\,\,\,\,\text{}\,\,\,\,y_1 \unlhd z, y_2 \unlhd z \Longrightarrow y_1+y_2 \unlhd z;.
\end{equation}
\begin{equation}\label{e3.5a}
z \unlhd y_1, z \unlhd y_2 \Longrightarrow z \unlhd y_1+y_2.
\end{equation}

For any $z \in P_\prec$ we define the sets
$$
F(z) :=\{y \in P_\prec \mid y \unlhd z\}\,\,\,\,\text{and}\,\,\,\,\widehat{F}(z) :=\{y \in P_\prec \mid y \lhd z\}.
$$
Since $z \in F(z)$ the set $F(z) \ne \varnothing$ for any $z \in P_\prec$. Moreover, it follows from \eqref{e3.5} that for every $z \in P_\prec$ the set $F(z)$ is a convex subcone of the cone $P_\prec$. As for the set $\widehat{F}(z)$ it can be empty for some $z \in P_\prec$. In the case when
$\widehat{F}(z) \ne \varnothing$ it is a cone which is, in general, not necessarily convex.

\begin{example}\label{ex3.5}
Let $\prec$ be the compatible partial preference on ${\mathbb{R}}^2$ with
$P_\prec = \{(z_1, z_2) \mid z_1 \ge 0, z_2 \ge 0\} \setminus \{(0,0)\}.$
Then $\widehat{F}(z)$ is empty for any $z =(z_1, z_2) \in P_\prec$ such that $z_1 > 0, z_2 =0$ or $z_1 = 0, z_2 >0$. For $z =(z_1, z_2) \in P_\prec$ with $z_1 > 0, z_2 >0$ the set $\widehat{F}(z)= \{(z_1, z_2) \mid z_1 > 0, z_2 =0\}\cup \{(z_1, z_2) \mid z_1 = 0, z_2 >0\}$ is a cone which is not convex.
\end{example}

However, when $\prec$ is a compatible weak preference, it follows from $\eqref{e3.4}$ that the majorization relation $\unlhd$ is total and, in this case, the set $\widehat{F}(z)$  is a convex cone as well (this will be proved below in Proposition \ref{cor3.7a}).

It follows from transitivity of the relation $\unlhd$ that $F(y) \subseteq F(z)$ whenever $y \unlhd z$. Evidently, the converse also is true. Thus, for any $y,z \in P_\prec$ the following identities are true:
\begin{equation*}
y \unlhd z \Longleftrightarrow F(y) \subseteq F(z), y \lhd z \Longleftrightarrow F(y) \subsetneq F(z),\,\,\text{\rm and}\,\,y \eq z \Longleftrightarrow F(y) = F(z).
\end{equation*}

The set $G(z):= \{y \in P_\prec \mid z \unlhd y\}$ also is a nonempty convex cone for every $z \in P_\prec$ (this follows from the implications \eqref{e3.5} and \eqref{e3.5a}) and, like the set $F(z)$, it can be used for characterization of $\unlhd$ on $P_\prec$, since for any $y,z \in P_\prec$ the following identities
\begin{equation*}
y \unlhd z \Longleftrightarrow G(y) \supseteq G(z), y \lhd z \Longleftrightarrow G(y) \supsetneq G(z),\,\,\text{\rm and}\,\,y \eq z \Longleftrightarrow G(y) = G(z)
\end{equation*}
hold.

Let $\mathcal{O}(P_\prec):= P_\prec/\eq$ be the set of equivalence classes of $P_\prec$ corresponding to the equivalence relation $\eq$. Clearly, that the positive cone $P_\prec$ is a disjoint union of all equivalence classes from $\mathcal{O}(P_\prec)$, i.e. $P_\prec = \bigcup\{E\mid E\in \mathcal{O}(P_\prec)\}$ and $E_1\bigcap E_2 = \varnothing$ for any $E_1,E_2 \in \mathcal{O}(P_\prec)$.
The equivalence class containing the vector $y$ will be denoted by $E_y$,  $E_y:=\{z \in P_\prec \mid z \eq y\}$.

\begin{proposition}\label{pr3.6}
Each equivalence class $E$ from $\mathcal{O}(P_\prec)$ is a relatively algebraic open convex cone.
\end{proposition}

\begin{proof}
Let $E \in {\mathcal{O}}(P_\prec)$ and let $z \in E$. Then $E =E_z := \{y \in P_\prec \mid z \eq y\}$. It is easily seen that $E_z = F(z) \bigcap G(z)$, where $F(z) := \{y \in P_\prec \mid y \unlhd z\}$ and
$G(z):= \{y \in P_\prec \mid z \unlhd y\}$. We have seen above that $F(z)$ and $G(z)$ are convex cones and, consequently, $E_z$ as the intersection of two convex cones is a convex cone too.

Note that for the convex cone $E_z$ its affine hull coincides with the vector subspace ${\rm Lin}(E_z) := E_z-E_z$. To prove that $z \in {\rm icr}E_z$ we need to show that for any $y \in {\rm Lin}(E_z)$ there is a real $\delta >0$ such that $z + ty \in P_\prec$ for all $t \in(-\delta,\delta)$. Choose an arbitrary $y \in {\rm Lin}(E_z)$ and $y_1,y_2 \in E_z$ for which $y = y_1 - y_2$. For $y_1$ and $y_2$ there exist reals $\nu_1 > \mu_1 > 0$, $\nu_2 > \mu_2 > 0$ such that $\mu_1 z \prec y_1 \prec \nu_1 z$ and $\mu_2 z \prec y_2 \prec \nu_2 z$. Setting $0 < \mu \le {\min\{\mu_1,\mu_2\}}$ and $\nu \ge {\max\{\nu_1,\nu_2\}}$ we obtain $\mu z \prec y_i \prec \nu z, i = 1,2,$ with $\mu < \nu$. From these inequalities, using the properties of $\prec$, we conclude that  $-(\nu-\mu)z \prec y_1  - y_2 \prec (\nu-\mu)z$ and, furthermore, that $(1-|t|(\nu-\mu))z \prec z +ty \prec (1+|t|(\nu-\mu))z$ for any real $t$. Since $1-|t|(\nu-\mu) > 0$ and $1+|t|(\nu-\mu) > 0$ for $t \in (-\delta, \delta)$ when $\delta = \displaystyle\frac{1}{\nu-\mu}$, then $z +ty \in E_z$ for all $t \in (-\delta, \delta)$. It proves that $z \in {\rm icr}E$. But $z$ is an arbitrary vector of $E$ and, hence, $E ={\rm icr}E$.
\end{proof}

Elements of $\mathcal{O}(P_\prec)$ will be called \textit{open components} of the positive cone $P_\prec$. As usual in the case of a quotient set, $\mathcal{O}(P_\prec)$ is assumed to be equipped with the partial order $\unlhd^*$ which is defined as follows: $E_1 \unlhd^* E_2$ holds for $E_1,E_2 \in \mathcal{O}(P_\prec)$ if and only if $y_1 \unlhd y_2$ for all (some) $y_1 \in E_1$ and all (some) $y_2 \in E_2$.

\begin{theorem}\label{th3.6}
The set $\mathcal{O}(P_\prec)$ of all open components of the positive cone $P_\prec$ equipped with the partial order $\unlhd^*$ is an upper lattice, that is, the partial ordered set $(\mathcal{O}(P_\prec), \unlhd^*)$ is such that each two-element subset $\{E_1,E_2\} \subset \mathcal{O}(P_\prec)$ has the least upper bound in $(\mathcal{O}(P_\prec), \unlhd^*)$, denoted by $E_1\vee E_2$, and, moreover, the equality $E_y \vee E_z = E_{y+z}$ holds for all $y,z \in P_\prec$.
\end{theorem}

\begin{proof}
Let $E_1,E_2 \in \mathcal{O}(P_\prec)$ and let $y \in E_1$ and $z \in E_2$, i.~e., $E_1 =E_y, E_2=E_z$. Since $\lambda y \prec y$ for all $\lambda \in (0,1)$ and $0 \prec z$, then $\lambda y \prec y+z$ for all $\lambda \in (0,1)$ and, consequently, $y \unlhd y+z.$ Similarly, we get $z \unlhd y+z$. Hence, $E_y \unlhd^* E_{y+z}$ and $E_z \unlhd^* E_{y+z}$.

Suppose now that $E_y \unlhd^* E$ and $E_z \unlhd^* E$ for some $E \in \mathcal{O}(P_\prec)$. Let $u \in E$. Then $y \unlhd u$ and $z \unlhd u$ or, equivalently, $\lambda y \prec u$ and $\mu z \prec u$ for some $\lambda >0, \mu >0$. From these inequalities we obtain that $\nu (y+z) \prec u$ for
$\nu = \displaystyle\frac{\min\{\lambda,\mu\}}2$. It proves that $y+z \unlhd u$ and, hence, $E_{y+z} \unlhd^* E$. Thus, $E_{y+z}=E_y \vee E_z$.
\end{proof}

\begin{corollary}\label{cor3.7}
Let $E \in \mathcal{O}(P_\prec)$ be an arbitrary open component of the positive cone $P_\prec$ and let ${\rm Lin}(E)= E-E$ be the linear hull of $E$ (the least vector subspace of $Y$ containing $E$). Then $${\rm Lin}(E)\bigcap P_\prec = \bigcup\{\tilde{E} \in \mathcal{O}(P_\prec) \mid \tilde{E} \unlhd^* E\}.$$
\end{corollary}

\begin{proof}
Let $\tilde{E},E \in \mathcal{O}(P_\prec)$ be such that $\tilde{E} \unlhd^* E$. Since $\tilde{E} \vee E =E$, then for any $y \in \tilde{E}$ and any $z \in E$ we have $w:= y + z \in E.$ Consequently, $y=w-z \in E-E={\rm Lin}(E)$. It proves that $\bigcup\{\tilde{E} \in \mathcal{O}(P_\prec) \mid \tilde{E} \unlhd^* E\} \subset {\rm Lin}(E)\bigcap P_\prec$.

Conversely, let $y \in {\rm Lin}(E)\bigcap P_\prec$. It follows from $y \in {\rm Lin}(E)$ that $y = z - w$ with $w,z \in E$. Since $z =y + w$, then $E_y\vee E= E$ and, consequently, $E_y \unlhd^* E$. It proves the converse inclusion ${\rm Lin}(E)\bigcap P_\prec \subset \bigcup\{\tilde{E} \in \mathcal{O}(P_\prec) \mid \tilde{E} \unlhd^* E\}$.
\end{proof}

Note that $\bigcup\{\tilde{E} \in \mathcal{O}(P_\prec) \mid \tilde{E} \unlhd^* E\}= F(z)$ for any $z \in E$.

\begin{corollary}\label{cor3.7a}
Let $\mathcal{E}$ be a subfamily of the family of open components $\mathcal{O}(P_\prec)$ of the positive cone $P_\prec$ that is linearly ordered by the relation $\unlhd^*$. For any $E \in \mathcal{E}$ which is not the least element of ${\mathcal{E}}$  the set $\widehat{F}_{\mathcal{E}}(E):= \bigcup\{\tilde{E} \subset \mathcal{E} \mid \tilde{E} \lhd^* E\}$ also is a  convex cone and ${\rm Lin}(\widehat{F}_{\mathcal{E}}(E))\bigcap E = \varnothing$, where $\tilde{E} \lhd^* E \Leftrightarrow \tilde{E} \unlhd^* E, \tilde{E} \ne E$ and ${\rm Lin}(\widehat{F}_{\mathcal{E}}(E))= \widehat{F}_{\mathcal{E}}(E)-\widehat{F}_{\mathcal{E}}(E)$ is the least vector subspace of $Y$ containing $\widehat{F}_{\mathcal{E}}(E)$.
\end{corollary}

\begin{proof}
Since the subfamily $\mathcal{E}$ is linearly ordered be $\unlhd^*$, then for any $y,z \in \widehat{F}_{\mathcal{E}}(E)$ we have either $E_{y+z} = E_y$ or $E_{y+z} = E_z$. Hence, $E_{y+z} \lhd^* E$ and $y+z \in \widehat{F}_{\mathcal{E}}(E)$. The positive homogeneity of $\widehat{F}_{\mathcal{E}}(E)$ follows from the homogeneity of each $E \in \mathcal{E}$. Thus, $\widehat{F}_{\mathcal{E}}(E)$ is a convex cone.

Assume now that ${\rm Lin}(\widehat{F}_{\mathcal{E}}(E))\bigcap E \ne \varnothing$ and let $z \in {\rm Lin}(\widehat{F}_{\mathcal{E}}(E))\cap E$. Through  $z \in {\rm Lin}(\widehat{F}_{\mathcal{E}}(E))$ we have $z = u - v$ with $u,v \in \widehat{F}_{\mathcal{E}}(E)$. It implies that $z+v=u \in \widehat{F}_{\mathcal{E}}(E)$. On the other hand, since $v \lhd z$, it follows from Proposition \ref{pr3.6} that $z+v \in E_z=E$. We get the contradiction to $E\bigcap \widehat{F}_{\mathcal{E}}(E) = \varnothing$. It proves that ${\rm Lin}(\widehat{F}_{\mathcal{E}}(E))\bigcap E = \varnothing$.
\end{proof}

Simple examples show that when the subfamily $\mathcal{E} \subset \mathcal{O}(P_\prec)$ is not linear ordered by $\unlhd^*$, the cone $\widehat{F}_{\mathcal{E}}(E)$ can be non-convex and $E \in \mathcal{E}$ can belong to ${\rm Lin}(\widehat{F}_{\mathcal{E}}(E))$.

\begin{example}[cf. with Example \ref{ex3.5}]
Let $Y={\mathbb{R}}^2$ and $P_\prec:= {\mathbb{R}}^2_+ \setminus \{(0,0)\}$. The set $\mathcal{O}(P_\prec)$ of open components of $P_\prec$ consists of three elements: $E_0:=\{(y_1,y_2)\mid y_1 >0,y_2>0\}$, $E_1:= \{(y_1,y_2)\mid y_1 >0,y_2=0\}$, and $E_2:=\{(y_1,y_2)\mid y_1 =0,y_2>0\}$. The subfamily $\mathcal{E} = \mathcal{O}(P_\prec)$ is not linearly ordered by $\unlhd^*$ and $\widehat{F}_{\mathcal{E}}(E_0)=\bigcup\{E \in \mathcal{E} \mid E \lhd^* E_0\} = E_1\bigcup E_2$ is non-convex cone with $E_0 \subset {\rm Lin}(\widehat{F}_{\mathcal{E}}(E_0)) = {\mathbb{R}}^2$.
\end{example}

\begin{theorem}
Let $(Y,\prec)$ be an ordered vector space. For the set $P_\pprec$ of strong positive vectors of $(Y,\prec)$ to be non-empty, it is necessary and sufficient that the set $\mathcal{O}(P_\prec)$ of open components of its positive cone $P_\prec$ ordered by the partial order $\unlhd^*$ have the greatest element, i.~e., such element $E_{\text{sup}} \in \mathcal{O}(P_\prec)$ that $E \unlhd^* E_{\text{sup}}$ for all $E \in \mathcal{O}(P_\prec)$. Moreover, in this case, $P_\pprec = E_{\text{sup}}$.
\end{theorem}

The proof of this theorem is immediate from the following observation:  a positive vector $y \in P_\prec$ is strong positive if and only if $z \unlhd y$ for all $z \in P_\prec$.

\smallskip

A compatible preference relation $\prec$ defined on a vector space $Y$ will be called \textit{relatively open} if ${\rm icr}P_\prec \ne \varnothing$ and $P_\prec = {\rm icr}P_\prec$ or, equivalently, if $P_\pprec \ne \varnothing$ and $P_\prec = P_\pprec$.

\begin{corollary}
A compatible preference relation $\prec$ defined on a real vector space $Y$ is relatively open if and only if the set $\mathcal{O}(P_\prec)$ of open components of its positive cone $P_\prec$ is a singleton, i.~e., if and only if $\mathcal{O}(P_\prec) = \{{\rm icr}P_\prec\} = \{P_\pprec\}$.
\end{corollary}

Let $Q$ be a convex set in a real vector space $Y$.

A nonempty convex subset $F \subset Q$ is called \textit{a face} of $Q$ (see, for instance, \cite{KhTZ}) if it satisfies the following property: if for some $u,v \in Q$ there exists $\alpha \in (0,1)$ such that $\alpha u + (1-\alpha)v \in F$ then $u,v \in F$. In other words, a nonempty convex subset $F \subset Q$ is \textit{a face} of $Q$ if every segment of $Q$, having in its relative interior an element of $F$, is entirely contained in $F$. The set $Q$ itself is its own face, and the empty set is considered as a face of any convex set $Q$. The face $F$ of $Q$ is said to be \textit{proper}if $F$ is nonempty and does not coincide with $Q$.

It is easy to see that when $Q$ is a convex cone, every face $F$ of $Q$ is a convex cone too.

\begin{theorem}[open components vs. faces]\label{th3.13}
Let $E \in \mathcal{O}(P_\prec)$ be an open component of the positive cone $P_\prec$.  Then the following assertions are true:

a$)$ The set $F(E):=\bigcup\{\tilde{E} \in \mathcal{O}(P_\prec) \mid \tilde{E} \unlhd^* E\}$ is a face of $P_\prec$ whose relative interior ${\rm icr}F(E)$ is not empty and, moreover, ${\rm icr}F(E) = E$.

b$)$ If the family $\mathcal{O}(P_\prec)$ is linearly ordered by $\unlhd^*$, the set $\widehat{F}(E):=\bigcup\{\tilde{E} \in \mathcal{O}(P_\prec) \mid \tilde{E} \lhd^* E\}$ also is a face of $P_\prec$.

\end{theorem}

\begin{proof}
a) Let $E \in \mathcal{O}(P_\prec)$. It is easily seen from Corollary \ref{cor3.7} that the set $F(E)$ is a convex cone. To prove that $F(E)$ is a face of $P_\prec$  we consider an arbitrary $y \in F(E)$ and assume that $y = \alpha u + (1 - \alpha)v$ for some $\alpha \in (0,1)$ and some  $u,v \in P_\prec$. Since $\alpha u \in E_u$ and $(1 - \alpha)v \in E_v$ then due to Theorem \ref{th3.6} $y \in E_{u+v}=E_u \vee E_v \unlhd^* E_y = E$ and, hence, $E_u \unlhd^* E$ and $E_v \unlhd^* E$. Consequently, $u \in E_u \subset F(E)$ and $v \in E_v \subset F(E)$. This proves that $F(E)$ is a face of $P_\prec$.

To prove the second part of the assertion a) we choose an arbitrary $y \in E$ and $z \in F(E)$. From $z \unlhd y$ we have that $y - \mu z \in P_\prec$ for some $\mu > 0$. Since $y - (y - \mu z) = \mu z \in F(E) \subset P_\prec$ then $y - \mu z \unlhd y$ and, hence, $y - \mu z \in F(E)$. Moreover, it is not difficult to see that $y - tz \in F(E)$ for all $t \in (0, \mu)$. From this and $ty \in F(E)\,\,\forall\,\,t>0$ we obtain $y - t(z-y) \in F(E)$ for all $t \in (0,\mu)$. Since both $z$ and $y$ belong to the convex cone $F(E)$ then $y+t(z-y) \in F(E)$ for all $t \in [0,1]$. Hence, $y+t(z-y) \in F(E)$ for all $t \in (-\delta,\delta)$, where $\delta = \min\{\mu,1\}$. Due to the arbitrary choice of the vector $z \in F(E)$ we conclude that $y \in {\rm icr}F(E)$. Hence, ${\rm icr}F(E) \ne \varnothing$ and, moreover, $E \subset {\rm icr}F(E)$.

Now, consider an arbitrary vector $u \in {\rm icr}F(E)$. Arguing as in the proof of Theorem~\ref{th3.2}, we obtain that $y \unlhd u$ for all $y \in F(E)$. From this, it easily follows that for  $z \in E$ we have $u \eq z$ and, hence, $u \in E$. Thus, ${\rm icr}F(E) \subset E$.

b) The case when $\widehat{F}(E) = \varnothing$ is trivial. Suppose that $\widehat{F}(E) \ne \varnothing$. With each  $\tilde{E} \in \mathcal{O}(P_\prec)$ such that $\tilde{E} \lhd^* E$ we associate the set $F(\tilde{E}) = \bigcup\{\hat{E} \in \mathcal{O}(P_\prec) \mid \hat{E} \unlhd^* \tilde{E}\}$ which is, through the assertion $a)$ proved above,  a face of $P_\prec$. It is easy to see that $\widehat{F}(E) = \bigcup\{F(\tilde{E}) \mid \tilde{E} \in \mathcal{O}(P_\prec), \tilde{E} \lhd^* E\}$ and, consequently, $\widehat{F}(E)$ is the union of the linearly ordered family of faces of $P_\prec$. By Proposition 2.8 from \cite{Millan} the set $\widehat{F}(E)$ also is a face of $P_\prec$.
\end{proof}

\begin{example}
Let $Y$ be a infinite-dimensional real vector space and let $\{e_i,i \in I\}$ be a Hamel basis for $Y$. Consider a compatible preference relation $\prec$ on $Y$ the positive cone $P_\prec$ of which is the convex cone consisting of non-zero vectors $y \in Y$ whose coordinates $\{y_i,i\in I\}$ in the given basis $\{e_i,i \in I\}$ are nonnegative. For $y \in P_\prec$ by $I(y)$ we denote the subset of those indices from $I$ for which $y_i > 0$, that is, $I(y) := \{i \in I \mid y_i > 0\}$. The subset $I(y)$ is non-empty and finite for all $y \in P_\prec$. It is not difficult to see that $y \unlhd z \Leftrightarrow I(y) \subseteq I(z)$ and $y \eq z \Leftrightarrow I(y) = I(z)$. We conclude from this that $E$ is an open component of a positive cone $P_\prec$, $E \in {\mathcal{O}}(P_\prec)$, if and only if $E = \{ y \in P_\prec \mid I(y) = J\}$, where $J$ is a some finite subset of $I$. Thus the cardinality of the family of the open components of the positive cone $P_\prec$ is equal to the cardinality of a Hamel basis of the vector space $Y$.
\end{example}

Let $\prec$ be a compatible partial preference and let $\mathcal{O}(P_\prec)$ be the family of open components of the positive cone $P_\prec$. We associate with each open component $E \in \mathcal{O}(P_\prec)$ the compatible partial preference $\prec_E$ on $Y$ such that $y_1 \prec_E y_2$ if and only if $y_2 - y_1 \in E$ and equip
the family $\mathcal{T}(\prec):=\{\prec_E \mid E \in \mathcal{O}(P_\prec)\}$ with the partial order $\unlhd'$ defined as follows: $\prec_{E_1} \unlhd' \prec_{E_2}$, for $\prec_{E_1},\prec_{E_2} \in \mathcal{T}(\prec)$, if and only if ${E_1} \unlhd^* {E_2}$

\begin{theorem}
For each compatible partial preference $\prec$ defined on $Y$ the family $\mathcal{T}(\prec):=\{\prec_E \mid E \in \mathcal{O}(P_\prec)\}$ consists of relatively open compatible partial preferences on $Y$, and with respect to the partial order $\unlhd'$ the family $\mathcal{T}(\prec)$ is an upper lattice with $\sup\{\prec_{E_1},\prec_{E_2}\} = \prec_{E_1 \vee E_2}$. Moreover, $y \prec z$ if and only if there exists $\prec_E \in {\mathcal{T}(\prec)}$ such that $y \prec_E z$.
\end{theorem}
\begin{proof}
The first claim immediately follows from Theorem~\ref{th3.6} since $(\mathcal{T}(\prec),\unlhd')$ is lattice isomorphic to $(\mathcal{O}(P_\prec),\unlhd^*)$. The last claim is a corollary of the equality $P_\prec = \bigcup\{E \mid E \in (\mathcal{O}(P_\prec)$ and the identity \eqref{e3}.
\end{proof}

Thus, the internal structure of a compatible partial preference relation $\prec$ is characterized by the structure of the upper lattice $(\mathcal{T}(\prec), \unlhd')$, elements of which are relatively open partial preference relations.

\section{Internal structure of compatible weak preference}

Unless otherwise stated, throughout this section we suppose that $\prec$ is a compatible weak preference defined on a real vector space $Y$. Recall, that a compatible partial preference $\prec$ is a compatible weak preference on $Y$ if and only if $Y = (-P_\prec)\cup L_\prec \cup P_\prec$, where $P_\prec$ is the positive cone of $\prec$ and $L_\prec$ is the vector space associated with $P_\prec$. Another equivalent characterization of a compatible weak preference $\prec$ is that $P_\prec$ is an asymmetric conical halfspace in $Y$. As above, the symbol $\unlhd$ denotes the majorization relation on the positive cone $P_\prec$ and $\mathcal{O}(P_\prec)$ is the family of open components of $P_\prec$ ordered by the partial order $\unlhd^*$ which is the factorization of $\unlhd$ by $\eq$.

\begin{proposition}
Let $\prec$ be a compatible weak preference on a real vector space $Y$.
Then the majorization relation $\unlhd$ defined on $P_\prec$ by $\prec$ is total.
\end{proposition}

\begin{proof}
Since $\prec$ is a weak preference on $Y$, one of the following three alternatives, $y \prec z$, $z \prec y$, $y \sim z$, holds for any $y,z \in P_\prec$.  Consequently, the assertion follows from the implications \eqref{e3.4}.
\end{proof}

\begin{corollary}
For a compatible weak preference $\prec$ the set $\mathcal{O}(P_\prec)$ of open components of $P_\prec$ is linearly ordered by $\unlhd^*$.
\end{corollary}

\begin{proof}
Let $E_1,E_2 \in \mathcal{O}(P_\prec)$ be such that $E_1 \ne E_2$. For $y \in E_1,z \in E_2$ the alternative $y \sim z$ is impossible because in this case, through \eqref{e3.4}, we would have $y \eq z$, which contradicts $E_1 \ne E_2$. Thus, we have either $y \prec z$ or $z \prec y$ and, consequently, since $E_1 \ne E_2$, either $E_1 \lhd^* E_2$ or $E_2 \lhd^* E_1$.
\end{proof}

\begin{proposition}\label{pr4.1a}
Let $\prec$ is a compatible weak preference on $Y$, $P_\prec$ its positive cone. Then for any $y,z \in P_\prec$ the following implication
$$
y \lhd z \Longrightarrow y \prec z.
$$
holds.
\end{proposition}

\begin{proof}
When $\prec$ is a weak preference on $Y$, for every pair $y,z \in Y$ exactly one of
$y \prec z,\,\,z \prec y$ or $y \sim z,$ holds.
But, through the inclusions \eqref{e3.4}, both $z \prec y$ and $y \sim z$
are impossible, since $y \lhd z$.
\end{proof}

\begin{proposition}\label{pr4.2a}
Let $\prec$ is a compatible weak preference on $Y$, $P_\prec$ its positive cone. Then for every open component $E \in {\mathcal O}(P_\prec)$ the inclusions
$$
L_\prec \subset L_E \subset {\rm Lin}(E)
$$
hold
\end{proposition}

\begin{proof}
It was shown in Subsection \ref{sec-2.2} that $L_K \subset {\rm Lin}(K)$ for any convex cone $K$. Thus, $L_E \subset {\rm Lin}(E)$.

Now, let $h \in L_\prec$. By the definition of $L_\prec$, we have
$$
y+th \in P_\prec\,\,\text{for all}\,\,y \in P_\prec\,\,\text{and all}\,\,t \in {\mathbb{R}}.
$$
Evidently,
$$
y+th \in P_\prec\,\,\text{for all}\,\,y \in E\,\,\text{and all}\,\,t \in {\mathbb{R}}.
$$
To prove the inclusion $L_\prec \subset L_E$ we need to show that actually
\begin{equation}\label{eA5}
y+th \in E\,\,\text{for all}\,\,y \in E\,\,\text{and all}\,\,t \in {\mathbb{R}}.
\end{equation}
To justify this, we argue by contradiction and suppose that there exist $\bar{y} \in E$ and $\bar{\lambda}\in {\mathbb{R}}$ such that $\bar{y}+\bar{\lambda}h \in \tilde{E}$ where $\tilde{E} \in {\mathcal O}(P_\prec),\,\tilde{E} \ne E$. Since ${\mathcal O}(P_\prec)$ is linearly ordered by $\lhd^*$, we have either $\tilde{E} \lhd^* E$ or $E \lhd^* \tilde{E}$. Assume that $\tilde{E} \lhd^* E$, then $\bar{y}+\bar{\lambda}h \lhd
\bar{y}$ and, through Proposition \ref{pr4.1a}, $\bar{y}+\bar{\lambda}h \prec \bar{y}$. Consequently, $\bar{y}-(\bar{y} +\bar{\lambda}h)= -\bar{\lambda}h \in P_\prec$. But this contradicts $P_\prec \cap L_\prec = \varnothing$. In the case $E \lhd^* \tilde{E}$ we would have $\bar{y} \lhd \bar{y}+\bar{\lambda}h$ and then $\bar{y} \prec \bar{y}+\bar{\lambda}h$. This implies $\bar{\lambda}h \in P_\prec$ that again contradicts $P_\prec \cap L_\prec = \varnothing$. Thus, any $h \in L_\prec$ satisfies \eqref{eA5} and, hence, $L_\prec \subset L_E$.
\end{proof}

\begin{theorem}\label{th4.3}
Let $\prec$ be a compatible weak preference on a real vector space $Y$. Then

1) for each open component  $E \in {\mathcal{O}}(P_\prec)$ the following equality holds:
\begin{equation}\label{e4.9a}
{\rm Lin}(E) = (-F(E))\cup L_\prec \cup F(E)
\end{equation}
where $F(E) := \bigcup\{\tilde{E} \in {\mathcal{O}}(P_\prec)\mid \tilde{E} \unlhd^* E\}$ and $L_\prec$ stands for the vector space associated with the positive cone $P_\prec$;

2) ${\rm Lin}(E_1) \subsetneq {\rm Lin}(E_2)$ for any $E_1,E_2 \in {\mathcal{O}}(P_\prec)$ such that $E_1 \lhd^* E_2$;

3) $Y =\bigcup \{{\rm Lin}(E)\mid E \in {\mathcal{O}}(P_\prec)\}.$
\end{theorem}

\begin{proof}
1) Since $Y=(-P_\prec)\cup L_\prec \cup P_\prec$ with $(-P_\prec) \cap P_\prec = \varnothing$ and $P_\prec \cap L_\prec = \varnothing$, then
${\rm Lin}(E)= {\rm Lin}(E)\cap Y = ({\rm Lin}(E)\cap (-P_\prec)) \cup ({\rm Lin}(E)\cap L_\prec)\cup ({\rm Lin}(E)\cap P_\prec)$. Now, the equality under proving follows from Corollary \ref{cor3.7}  and the inclusion $L_\prec \subset{\rm Lin}(E)$.

2) It follows immediately from the definition of the set $F(E)$ that $F(E_1) \subsetneq F(E_2)$ holds for any $E_1,E_2 \in {\mathcal{O}}(P_\prec)$ such that $E_1 \lhd^* E_2$. Thus, the inclusion we need to prove is an obvious consequence of the equality proved in 1).

3) To prove this equality we observe that each $y \in Y$ belongs either $L_\prec$ or $F(E)\cup (-F(E))$ for some $E \in {\mathcal{O}}(P_\prec)$.
\end{proof}

\begin{corollary}
Let $\prec$ be a compatible weak preference on a real vector space $Y$, and let $E_{inf}$ be the least element of the linearly ordered set $({\mathcal{O}}(P_\prec), \unlhd^*)$ then $${\rm Lin}(E_{inf}) = (-E_{inf})\cup L_\prec \cup E_{inf},$$
and, in addition, $L_{E_{inf}} = L_\prec$.
\end{corollary}

\begin{proof}
The first equality we need to prove follows from the assertion 1) of the above proposition, since $F(E)=E$ when $E=E_{inf}$ is the least element of $({\mathcal{O}}(P_\prec), \unlhd^*)$.

Further, since $E \cap L_E = \varnothing$ and $L_E \subset {\rm Lin}(E)$ for all $E \in {\mathcal{O}}(P_\prec)$, it follows from the equality we just proved that $L_{E_{inf}} \subset L_\prec$. But, through Proposition \ref{pr4.2a}, we have $L_\prec \subset L_{E_{inf}}$ and, hence, $L_{E_{inf}} = L_\prec$.
\end{proof}

\begin{proposition}
Let $\prec$ be a compatible weak preference on a real vector space $Y$ and let $E \in {\mathcal{O}}(P_\prec)$ be an open component of its positive cone $P_\prec$. Suppose that $E$ is not the least element of the linearly ordered set $({\mathcal{O}}(P_\prec), \unlhd^*)$,
then the linear hull of the convex cone $\widehat{F}(E) := \bigcup\{\tilde{E} \in {\mathcal{O}}(P_\prec)\mid \tilde{E} \lhd^* E\}$ can be presented as follows:
\begin{equation}\label{e4.9}
{\rm Lin}(\widehat{F}(E))= (-\widehat{F}(E))\cup L_\prec \cup \widehat{F}(E),
\end{equation}
and
\begin{equation}\label{e4.10}
{\rm Lin}(E)= (-E)\cup {\rm Lin}(\widehat{F}(E))\cup E.
\end{equation}
\end{proposition}

\begin{proof}
Since $E$ is not the least element of the linearly ordered set $({\mathcal{O}}(P_\prec), \unlhd^*)$, the family $\{\tilde{E} \in {\mathcal{O}}(P_\prec) \mid \tilde{E} \lhd^* E\}$ is nonempty and, consequently, the cone $\widehat{F}(E)$ is also nonempty.
Note that the nonempty family of linear subspaces $\{{\rm Lin}(\tilde{E}) \mid \tilde{E} \in {\mathcal{O}}(P_\prec), \tilde{E} \lhd^* E\}$ is linearly ordered by inclusion and hence $\bigcup\{{\rm Lin}(\tilde{E}) \mid \tilde{E} \in {\mathcal{O}}(P_\prec), \tilde{E} \lhd^* E\}$ also is a linear subspace. Furthermore, since $\tilde{E} \subset \widehat{F}(E)$ for every $\tilde{E} \in {\mathcal{O}}(P_\prec), \tilde{E} \lhd^* E$ then ${\rm Lin}(\tilde{E}) \subset {\rm Lin}(\widehat{F}(E))$ and consequently $\bigcup \{{\rm Lin}(\tilde{E})\mid \tilde{E} \in {\mathcal{O}}(P_\prec)\mid \tilde{E} \lhd^* E\} \subset {\rm Lin}(\widehat{F}(E))$. On the other hand, the converse inclusion follows from $\widehat{F}(E) \subset \bigcup \{{\rm Lin}(\tilde{E})\mid \tilde{E} \in {\mathcal{O}}(P_\prec)\mid \tilde{E} \lhd^* E\}$ and thus we have
\begin{equation}\label{e4.10a}
{\rm Lin}(\widehat{F}(E)) = \bigcup \{{\rm Lin}(\tilde{E})\mid \tilde{E} \in {\mathcal{O}}(P_\prec), \tilde{E} \lhd^* E\}.
\end{equation}
Further, using the equality 1) from Theorem \ref{th4.3} we obtain ${\rm Lin}(\widehat{F}(E)) = \bigcup \{(-F(\tilde{E}))\cup L_\prec \cup F(\tilde{E})\mid \tilde{E} \in {\mathcal{O}}(P_\prec), \tilde{E} \lhd^* E\} = (-\widehat{F}(E))\cup L_\prec \cup \widehat{F}(E)$. Thus, we proved the equality \eqref{e4.9}.

Since $F(E) = E\cup \widehat{F}(E)$, the equality \eqref{e4.10} immediately follows from the equality ${\rm Lin}(E) = (-F(E))\cup L_\prec \cup F(E)$ (see Theorem \ref{th4.3}) and from the equality \eqref{e4.9}.
\end{proof}

\begin{corollary}
Let a compatible weak preference $\prec$ be such that the set ${\mathcal{O}}(P_\prec)$ of open components of its positive cone $P_\prec$ is finite. Suppose that ${\mathcal{O}}(P_\prec) = \{E_1,E_2,\ldots,E_m\}$ and $E_1 \lhd^* E_2 \lhd^* \ldots \lhd^* E_m$. Then
$$
{\rm Lin}(E_i) = (-E_i)\cup(-E_{i-1})\cup\ldots\cup (-E_1)\cup L_\prec \cup E_1 \cup \ldots \cup E_{i-1} \cup E_i\,\,\text{for all}\,\,i=1,2,\ldots,m,
$$
in particular,
$$
Y = {\rm Lin}(E_m) = (-E_m)\cup(-E_{m-1})\cup\ldots\cup (-E_1)\cup L_\prec \cup E_1 \cup \ldots \cup E_{m-1} \cup E_m.
$$
\end{corollary}

For the proof of the next proposition we need the following lemma.

\begin{lemma}\label{l4.1}
Let $Z$ be a vector subspace in a vector space $Y$.
An asymmetric convex cone $H \subset Z$ is a conical halfspace in $Z$ (this means that both $H$ and $Z \setminus H$ are convex cones), if and only if the set $Z \setminus (H\bigcup(-H))$ is a vector subspace. Moreover, the vector subspace $L_H$ associated with  an asymmetric conical halfspace $H \subset Z$ coincides with the vector subspace $Z \setminus (H\bigcup(-H))$.
\end{lemma}

A counterpart of the first assertion of this lemma can be found in \cite[see Corollary, p.165]{GS98}.
The proof of the second assertion are straightforward and we leave it for readers.

\begin{proposition}
Let $\prec$ be a compatible weak preference on a real vector space $Y$ and let ${\mathcal{O}}(P_\prec)$ be the set of open components of the positive cone $P_\prec$.
Then

i) each open component $E \in {\mathcal{O}}(P_\prec)$ is a relatively open halfspaces in ${\rm Lin}(E)$ and
\begin{equation}\label{e4.11}
L_E = {\rm Lin}(\widehat{F}(E)) = (-\widehat{F}(E))\cup L_\prec \cup \widehat{F}(E);
\end{equation}

ii) for each open component $E \in {\mathcal{O}}(P_\prec)$ the set $F(E)$ (which by Theorem \ref{th3.13} is a face of $P_\prec$) is a conical halfspace in ${\rm Lin}(E)$ and $L_{F(E)} = L_\prec$;

iii) for each open component $E \in {\mathcal{O}}(P_\prec)$ the set $\widehat{F}(E)$ (which  by Theorem \ref{th3.13} is a face of $P_\prec$) is a conical halfspace in $L_E$ and $L_{\widehat{F}(E))} = L_\prec$.

{\rm Recall that $L_K := \{h \in Y \mid y + \alpha h \in K\,\,\forall\,y \in K\,\,\text{and}\,\,\forall\,\alpha \in {\mathbb{R}}\}$ is the vector subspace associated with the convex cone $K$ while ${\rm Lin}(K)$ is the convex hull of $K$, and $L_\prec = L_{P_\prec}$.
}
\end{proposition}

\begin{proof}
The validity of the assertions $i$), $ii$) and $iii$) follows immediately from Lemma \ref{l4.1} and the equalities \eqref{e4.10}, \eqref{e4.9a} and \eqref{e4.9}, respectively.
\end{proof}

\begin{corollary}
Let $\prec$ be a compatible weak preference on $Y$. Then for every open component $E \in {\mathcal{O}}(P_\prec)$ the equality
\begin{equation*}
{\rm Lin}(E) = (-E)\cup L_E \cup E
\end{equation*}
holds and, consequently, each $E \in {\mathcal{O}}(P_\prec)$ is a relatively open asymmetric conical halfspace in ${\rm Lin}(E)$.
\end{corollary}

\begin{remark}
It follows from \eqref{e4.10}, \eqref{e4.10a} and \eqref{e4.11} that
\begin{equation}\label{e4.17}
L_E = \bigcup\{{\rm Lin}(\tilde{E}) \mid \tilde{E} \in {\mathcal O}(P_\prec), \tilde{E} \lhd^*E\}
\end{equation}
for each $E \in {\mathcal O}(P_\prec)$ .

Taking into account  the assertion 2) from Theorem \ref{th4.3} we see that in the case when the family $\{\tilde{E} \in {\mathcal O}(P_\prec), \tilde{E} \lhd^*E\}$ has the greatest element $\widehat{E}$ with respect to $\lhd^*$ the vector subspace $L_E$ coincides with ${\rm Lin}(\widehat{E})$, i.e., $L_E = {\rm Lin}(\widehat{E})$.
\end{remark}

Recall (see Section 3) that we associate with each open component $E \in \mathcal{O}(P_\prec)$ the compatible partial preference $\prec_E$ on $Y$ such that $y_1 \prec_E y_2$ if and only if $y_2 - y_1 \in E$ and equip
the family $\mathcal{T}(\prec):=\{\prec_E \mid E \in \mathcal{O}(P_\prec)\}$ with the partial order $\unlhd'$ defined as follows: $\prec_{E_1} \unlhd' \prec_{E_2}$, for $\prec_{E_1},\prec_{E_2} \in \mathcal{T}(\prec)$, if and only if ${E_1} \unlhd^* {E_2}$.

\begin{theorem}
For each compatible weak preference $\prec$ defined on $Y$ the family \linebreak $\mathcal{T}(\prec):=\{\prec_E \mid E \in \mathcal{O}(P_\prec)\}$ is linearly ordered by $\unlhd'$ and consists of such relatively open compatible partial preferences $\prec_E$ on $Y$ the restriction of which to ${\rm Lin}(E)$ is a weak preference.
\end{theorem}

The proof follows from the preceding assertions.

\smallskip

The last theorem shows that the internal structure of a compatible weak preference $\prec$ can be identified with the structure of the linearly ordered the family $(\mathcal{T}(\prec), \unlhd')$ and, consequently, it is more refined in comparison with the internal structure of an arbitrary compatible partial preference, for which the family $(\mathcal{T}(\prec), \unlhd')$ is, in general, only an upper lattice.

\smallskip

Before to complete this section we prove a number of assertions which will used in the next section for to obtain the analytical representation of compatible weak preferences.

\begin{proposition}\label{pr4.11}
Let $\prec$ be a compatible weak preference on $Y$. For each open component $E \in {\mathcal{O}}(P_\prec)$ there is a nonzero linear functions $\phi_E: Y \to {\mathbb{R}}$ such that
\begin{equation}\label{e4.18}
E =\{y \in {\rm Lin}(E) \mid \phi_E(y) > 0\}\,\,\text{and}\,\,
L_E = \{y \in {\rm Lin}(E) \mid \phi_E(y) = 0\}.
\end{equation}
Furthermore, $\phi_E(y) = 0$ for all $y \in {\rm Lin}(\tilde{E})$ provided that $\tilde{E} \in {\mathcal O}(P_\prec), \tilde{E} \lhd^* E$.
\end{proposition}

\begin{proof}
Note that $E$ and $L_E$ are convex subsets in ${\rm Lin}(E)$ and $E \cap L_E = \varnothing$. Furthermore, $E$ is relatively algebraic open.  By Corollary 5.62 from \cite{Aliprantis} there is a nonzero linear functional $\phi_E$ defined on $Y$ which properly separates $E$ and $L_E$. It is not hard to verify that $\phi_E$ satisfies the equalities \eqref{e4.18}.

The last assertion follows from the second equality in \eqref{e4.18} and the equality \eqref{e4.17}.
\end{proof}

Thus, every compatible weak preference $\prec\, \subset Y \times Y$ can be associated with the following family ${\mathcal F}_\prec$ of nonzero linear functions:
$$
{\mathcal F}_\prec := \{\phi_E \in {\mathcal L}(Y,{\mathbb{R}}) \mid E \in {\mathcal O}(P_\prec), \phi_E\,\,\text{satisfies}\,\,\eqref{e4.18}\},
$$
where ${\mathcal L}(Y,{\mathbb{R}})$ stands for the vector space of linear functions defined on $Y$.

Notice that the family ${\mathcal F}_\prec$ is non uniquely  defined.

We endow
the family ${\mathcal F}_\prec$ with the linear order $\unlhd\,'$ defined as follows:
$$
\phi_E \unlhd' \phi_{\tilde{E}} \Longleftrightarrow {\tilde{E}} \unlhd^* E.
$$
Clearly $({\mathcal F}_\prec, \unlhd\,')$ is order anti-isomorphic to $({\mathcal O}(P_\prec),\unlhd^*)$.

\begin{proposition}
The family ${\mathcal F}_\prec$ is linearly independent.
\end{proposition}

\begin{proof}
Arguing by contradiction, assume that there is a finite subfamily $\{\phi_{E_1},\phi_{E_2},\ldots,\phi_{E_k}\} \subset {\mathcal F}_\prec$ and nonzero real numbers $\alpha_1,\alpha_2,\ldots,\alpha_k$ such that $\alpha_1\phi_{E_1}(y)+\alpha_2\phi_{E_2}(y)+\ldots+\alpha_k\phi_{E_k}(y)=0$ for all $y \in Y$. Without loss of generality we may assume that $\phi_{E_1} \lhd'\phi_{E_2} \lhd' \ldots \lhd' \phi_{E_k}$. Then $E_k \lhd^* E_{k-1} \lhd^*\ldots \lhd^* E_2 \lhd^* E_1$ and hence, through the second assertion of Proposition \ref{pr4.11}, for every $i = 1,2,\ldots,k-1$ the equality $\phi_{E_i}(y)=0$ holds for all $y \in {\rm Lin}(E_k)$. Since $-\alpha_k\phi_{E_k}(y)=\alpha_1\phi_{E_1}(y)+\alpha_2\phi_{E_2}(y)+\ldots+\alpha_{k-1}\phi_{E_{k-1}}(y)$ for all $y \in Y$, we have $\phi_{E_k}(y)=0$ for all $y \in {\rm Lin}(E_k)$, but this contradicts $\phi_{E_k}(y)>0$ for all $y \in E_k$.
\end{proof}

\begin{proposition}
Let $\prec$ be a compatible weak preference and let ${\mathcal F}_\prec = \{\phi_E \mid E \in {\mathcal O}(P_\prec)\}$  be a family of linear functions associated with $\prec$.
For any open component ${E} \in {\mathcal O}(P_\prec)$, which is not the greatest element of ${\mathcal O}(P_\prec)$, we have
\begin{equation}\label{e4.19}
{\rm Lin}({E})=\{y \in Y \mid \phi_{\tilde{E}}(y)=0\,\,\text{for all}\,\,\tilde{E} \in {\mathcal O}(P_\prec)\,\,\text{such that}\,\,{E} \lhd^* \tilde{E}\},
\end{equation}
and ${\rm Lin}({E})= Y$, when ${E}$ is the greatest element of ${\mathcal O}(P_\prec)$.
\end{proposition}

\begin{proof}
Suppose that ${E} \in {\mathcal O}(P_\prec)$ is not the greatest element of ${\mathcal O}(P_\prec)$.  Then the family $\{\tilde{E} \in {\mathcal O}(P_\prec) \mid {E} \lhd^* \tilde{E}\}$ is nonempty and we have, through the second assertion of Proposition \ref{pr4.11}, that $\phi_{\tilde{E}}(y)=0$ for all $y \in {\rm Lin}({E})$ whenever $\tilde{E} \in {\mathcal O}(P_\prec)$ is such that ${E} \lhd^* \tilde{E}$. Thus, the implication
$$
y \in {\rm Lin}({E}) \Longrightarrow \phi_{\tilde{E}}(y)=0\,\,\text{for all}\,\,\tilde{E} \in {\mathcal O}(P_\prec)\,\,\text{such that}\,\,{E} \lhd^* \tilde{E}
$$
is true.

Let us show that the converse implication is also true. Suppose that for some $y \in Y$ the equality $\phi_{\tilde{E}}(y) = 0$ holds for every $\tilde{E} \in {\mathcal O}(P_\prec)$ such that ${E} \lhd^* \tilde{E}$. Then it follows from \eqref{e4.18} that $y \notin (-{\tilde{E}})\cup {\tilde{E}}$ for all $\tilde{E} \in {\mathcal O}(P_\prec)$ such that ${E} \lhd^* \tilde{E}$ and hence $y~\in~(-{E})\bigcup (-{\widehat{F}}({E}) \bigcup L_\prec \bigcup {\widehat{F}}({E}) \bigcup {E}  = {\rm Lin}({E}).$

It completes the proof of the first assertion.

The validity of the equality $Y={\rm Lin}(E)$ for the greatest element $E \in ({\mathcal O}(P_\prec),\unlhd)$ follows from Theorem \ref{th4.3}, the assertions 2) and 3).
\end{proof}

\begin{proposition}\label{pr4.17}
Let $\prec$ be a compatible weak preference and let ${\mathcal F}_\prec = \{\phi_E \mid E \in {\mathcal O}(P_\prec)\}$  be the family of linear functions associated with $\prec$. Then for each  $E \in {\mathcal O}(P_\prec)$ which is not the greatest element of $({\mathcal O}(P_\prec), \unlhd^*)$ the following equality
\begin{equation}\label{e4.21}
E= \{y \in Y \mid \phi_{\tilde{E}}(y)=0\,\,\text{for all}\,\,\tilde{E} \in {\mathcal O}(P_\prec)\,\,
\text{such that}\,\,E \lhd^* \tilde{E}\,\,\text{and}\,\,\phi_E(y) > 0\}
\end{equation}
is true, while
\begin{equation}\label{e4.21a}
E= \{y \in Y \mid \phi_E(y) > 0\}
\end{equation}
when $E \in {\mathcal O}(P_\prec)$ is the greatest element of $({\mathcal O}(P_\prec), \unlhd^*)$.

Furthermore, the equality
\begin{equation}\label{e4.22}
L_E= \{y \in Y \mid \phi_{\tilde{E}}(y)=0\,\,\text{for all}\,\,\tilde{E} \in {\mathcal O}(P_\prec)\,\,
\text{such that}\,\,E \unlhd^* \tilde{E}\}
\end{equation}
holds for every $E \in {\mathcal O}(P_\prec)$.

At last,
\begin{equation}\label{e4.23}
L_\prec= \{y \in Y \mid \phi_{{E}}(y)=0\,\,\forall\,\,{E} \in {\mathcal O}(P_\prec)\}.
\end{equation}
\end{proposition}
\begin{proof}
The equalities \eqref{e4.21} and \eqref{e4.22} follows from \eqref{e4.18} and \eqref{e4.19}. Since ${\rm Lin}(E) = Y$ when $E \in {\mathcal O}(P_\prec)$ is the greatest element of $({\mathcal O}(P_\prec), \unlhd^*)$ the equality \eqref{e4.21a} is actually the first equality from \eqref{e4.18}.

The equality \eqref{e4.23} follows from \eqref{e4.22} and the inclusion $L_\prec \subset L_E$ which holds for every $E \in  {\mathcal O}(P_\prec)$.
\end{proof}

\begin{theorem}\label{th4.18}
Let $\prec$ be a compatible weak preference and let ${\mathcal F}_\prec = \{\phi_E \mid E \in {\mathcal O}(P_\prec)\}$  be a family of linear functions associated with $\prec$.
For each $y \in Y$ either the subfamily ${\mathcal F}_\prec(y):=\{\phi_E \in {\mathcal F}_\prec \mid \phi_E(y) \ne 0\}$ is empty or ${\mathcal F}_\prec(y)$ has the least (with respect to $\unlhd\,'$) element.

Conversely, for each $\phi_E \in {\mathcal F}_\prec$ there exists $y \in Y$ such that $\phi_E$ is the least (with respect to $\unlhd\,'$) element of ${\mathcal F}_\prec(y)$.
\end{theorem}

\begin{proof}
Since $L_\prec \subset L_E \subset {\rm Lin}(E)$ for all $E \in {\mathcal O}(P_\prec)$ (see Proposition \ref{pr4.2a}), then, through the second equality in \eqref{e4.19}, for every $E \in {\mathcal O}(P_\prec)$ we have $\phi_E(y)=0$ for all $y \in L_\prec$. Thus, ${\mathcal F}_\prec(y) = \varnothing$ for all $y \in L_\prec$.

Now, let $y \in Y \setminus L_\prec$. Then, there exists $E \in {\mathcal O}(P_\prec)$ such that either $y \in E$ or $y \in -E$ and hence, due to the first equality in \eqref{e4.18}, we have either $\phi_E(y) > 0$  or $\phi_E(y) < 0$. Consequently, $\phi_E(y) \ne 0$ for all $y \in Y \setminus L_\prec$. Moreover, it follows from \eqref{e4.21} that $\phi_{\tilde{E}}(y)=0$ for all $\phi_{\tilde{E}} \lhd' \phi_E$. Thus, $\phi_E$ is the least element of ${\mathcal F}_\prec(y)$.

Conversely, it is easy to verify that each  $\phi_E \in {\mathcal F}_\prec$ is the least element of ${\mathcal F}_\prec(y)$ with $y \in E$.
\end{proof}

\section{Analytical representation of compatible weak preferences}

\setcounter{theorem}{0} 

Let $Y$ be a real vector space and let ${\mathcal L}(Y,{\mathbb{R}})$ be a vector space of linear functions defined on $Y$.

\begin{definition}\cite{Gor-m}
We say that a family of linear functions ${\mathcal F} \subset {\mathcal L}(Y,{\mathbb{R}})$ forms \textit{a cortege of linear functions on $Y$} if

$(i)$  ${\mathcal F}$ is ordered by some linear order $\unlhd_{{\mathcal F}}$.;

$(ii)$ for each $y \in Y$ either the subfamily ${\mathcal F}_y := \{\phi \in {\mathcal F} \mid \phi(y) \ne 0\}$ is empty or ${\mathcal F}_y$ has the least (with respect to $\unlhd_{{\mathcal F}}$) element $\phi_y$;

$(iii)$ for each $\phi \in {\mathcal F}$ there exists $y \in Y$ such that  $\phi=\phi_y$, i.e., each $\phi \in {\mathcal F}$ is the least element ${\mathcal F}_y$ for some $y \in Y$.
\end{definition}
Families of linear functions satisfying conditions $(i)$ and $(ii)$ were introduced by
Klee \cite{Klee} for analytic representation of semispaces. The term ``cortege of linear functions'' for families of linear function satisfying the conditions ($i$) -- ($iii$) was introduced
by Gorokhovik \cite{Gor-m} (for various versions of condition ($iii$)) see \cite{Gor-m,GS97,GS98a,GS99,GS2000,Gor21}). In \cite{GS99,GS2000} Gorokhovik and Shinkevich introduced the concept of a cortege of affine functions generalizing the concept of a cortege of linear functions and used it for the analytical representation of arbitrary halfspaces in infinite-dimensional vector spaces.

Let $({\mathcal F},\unlhd_{\mathcal F})$ be a cortege of linear functionals on $Y$.
For each $\phi \in {\mathcal F}$ we denote by $Y_\phi$ and $\widehat{Y}_\phi$ the following vector subspaces
$$
Y_\phi := \{y \in Y \mid \phi'(y)=0\,\,\text{for all}\,\,\phi' \lhd_{\mathcal F} \phi\}
$$
and
$$
\widehat{Y}_\phi := \{y \in Y \mid \phi'(y)=0\,\,\text{for all}\,\,\phi' \unlhd_{\mathcal F} \phi\}.
$$
If $\phi \in {\mathcal F}$ is the least element of ${\mathcal F}$ we suppose $Y_\phi =Y$.

It follows from $(ii)$ that if $y \in Y$ is such that ${\mathcal F}_y \ne \varnothing$ then $y$ belongs $Y_{\phi_y}$. It implies that $\phi_y(Y_{\phi_y})= {\mathbb R}$  and, through $(iii)$, we conclude that every
$\phi \in {\mathcal F}$ is nonzero. Moreover, the following claim is true.
\begin{proposition}
Let $({\mathcal F}, \unlhd_{\mathcal F})$ be a cortege of linear functions on $Y$. Then ${\mathcal F}$ is a linearly independent family in ${\mathcal L}(Y,{\mathbb{R}})$.
\end{proposition}

\begin{proof}
Let ${\mathcal F}$ be a cortege of linear functions on $Y$ and let $\{\phi_1,\phi_2,\ldots,\phi_k\}$ be an arbitrary finite subfamily of ${\mathcal F}$. Without loss of generality we can assume that $\phi_1 \lhd_{\mathcal F} \phi_2 \lhd_{\mathcal F} \ldots \lhd_{\mathcal F} \phi_k.$ Suppose that for reals $\mu_1,\mu_2,\ldots,\mu_k$ the equality $\mu_1\phi_1 + \mu_2\phi_2+\ldots+\mu_k\phi_k=0$ holds. It follows from the property $(iii)$ of a cortege that for $\phi_k \in {\mathcal F}$ there exists a vector $y_k \in Y$ such that $\phi_k = \phi_{y_k}$ and, consequently, $\phi_k(y_k) \ne 0$. Since $f_k=f_{y_k}$ is the least element of ${\mathcal F}_{y_k}$, then for $f_1,f_2,\ldots,f_{k-1}$ the equalities $f_1(y_k) = f_2(y_k)=\ldots=f_{k-1}(y_k) =0$ hold. This implies $\sum_{i=1}^{k}\mu_i\phi_i(y_k) = \mu_k\phi_k(y_k)=0$. Since $\phi_k(y_k)\ne 0$ we obtain that $\mu_k =0$ and hence $\mu_1\phi_1 + \mu_2\phi_2+\ldots+\mu_{k-1}\phi_{k-1}=0$. Repeating the same arguments with respect to $\phi_{k-1}$ we obtain that $\mu_{k-1} = 0$. Continuing this process further we will get in $k$ steps  $\mu_1=0,\mu_2=0,\ldots,\mu_k=0$. It proves that ${\mathcal F}$ is linearly independent.
\end{proof}

\begin{definition}\cite{Gor-m,GS97,GS2000,Gor21}
A real-valued function $u:Y \to {\mathbb{R}}$ is called \textit{step-linear} if there exists a cortege of linear functions ${\mathcal F}$ such that $u(y) = u_{{\mathcal F}}(y), y \in Y$, where
$$
u_{\mathcal F}(y):=\left\{
  \begin{array}{cr}
       0,&\text{when}\,\,{{\mathcal F}}_{y} = \varnothing,\\
     \phi_y(y),&\text{when}\,\,{{\mathcal F}}_{y} \ne \varnothing.\\
  \end{array}
     \right.
$$
\end{definition}

Clearly, that any nonzero linear function $\phi$ can be considered as a step-linear one generated by the one-element cortege ${\mathcal F} =\{\phi\}$.

For a finite cortege of linear functions ${\mathcal F}:= \{\phi_1,\phi_2,\ldots,\phi_k\}$ which is ordered according to numbering ($\phi_1 \lhd_{\mathcal F} \phi_2 \lhd_{\mathcal F} \ldots \lhd_{\mathcal F} \phi_k$) the function $u_{\mathcal F}$ is defined as follows
\begin{equation*}
u_{{\mathcal F}}(y)=\left\{
  \begin{array}{ll}
       \phi_1(y),&\text{if}\,\,\phi_1(y) \ne 0,\\
       \phi_2(y),&\text{if}\,\,\phi_1(y) = 0,\phi_2(y) \ne 0,\\
       ......... & ....................................................................... \\
       \phi_{k-1}(y),&\text{if}\,\,\phi_1(y) = 0, \ldots,  \phi_{k-2}(y)=0,\phi_{k-1}(y) \ne 0,\\
       \phi_{k}(y),&\text{if}\,\,\phi_1(y) =0, \ldots, = \phi_{k-1}(y)=0.
       \end{array}
     \right.
\end{equation*}

For the first time step-linear functions (under the name ``conditionally linear functions'') were introduced in \cite{Gor-m}. Later step-linear functions as well as step-affine ones generalized them were studied in \cite{GS97,GS98a,GS99,GS2000,Gor21}.

\begin{proposition}
Every step-linear function $u:Y \to {\mathbb{R}}$ is homogeneous, i.e., $u(ty)= tu(y)$ for all $y \in Y$ and all $t \in {\mathbb{R}}$, and, moreover, for any $y,z \in Y$ the following implications hold:
\begin{equation}\label{e5.22}
u(z)=0 \Longrightarrow u(y+z)=u(y)
\end{equation}
and
\begin{equation}\label{e5.23}
u(y)>0,u(z)>0 \Longrightarrow u(y+z)> 0.
\end{equation}
\end{proposition}

\begin{proof}
Let $u:Y \to {\mathbb{R}}$ be a step-linear function and let ${\mathcal F}$ be a cortege of linear functions on $Y$ such that $u(y) = u_{\mathcal F}(y)$ for all $y \in Y$.

The homogeneity of $u$ follows from the homogeneity of the linear functions belonging to the cortege ${\mathcal F}$.

To prove the implications \eqref{e5.22} and \eqref{e5.23} we first observe that through additivity of linear functions  the equality
$\phi(y+z) = \phi(y)+\phi(z)$
holds for any $y,z \in Y$ and for all $\phi \in {\mathcal F}$.

If $u(z)=0$ then $\phi(z) = 0$ for all $\phi \in {\mathcal F}$ and hence  $\phi(y+z) = \phi(y)$ for all $\phi \in {\mathcal F}$. This proves \eqref{e5.22}.

Assume that $u(y)>0,u(z)>0$. If $\phi_y \lhd_{\mathcal{F}} \phi_z$ then $\phi(y+z)=\phi(y)+\phi(z)=0$ for all $\phi \lhd_{\mathcal{F}} \phi_y$ and $\phi_y(y+z)=\phi_y(y)+\phi_y(z) = \phi_y(y)$. Thus, $\phi_{y+z}=\phi_y$ and hence $u(y+z)=u(y)>0$. Similarly, if $\phi_y \lhd_{\mathcal{F}} \phi_z$ we get that $u(y+z)=u(z)>0$. At last, if $\phi_y = \phi_z$, we have that $\phi(y+z)=\phi(y)+\phi(z)=0$ for all $\phi \lhd_{\mathcal{F}} \phi_y=\phi_z$ and $\phi_y(y+z)=\phi_y(y)+\phi_y(z)=\phi_y(y)+\phi_z(z)>0$. Thus, in this case we have that $\phi_{y+z}=\phi_y=\phi_z$ and $u(y+z)=u(y)+u(z)>0$. This completes the proof of the implication \eqref{e5.23}.
\end{proof}

\begin{theorem}\label{th5.7}
A binary relation $\prec \,\subset Y \times Y$ is a compatible weak preference on $Y$ if and only if there exists a step-linear function $u:Y \to {\mathbb{R}}$ such that
\begin{equation}\label{e5.27}
y \prec z \Longleftrightarrow u(z-y) > 0
\end{equation}
and
\begin{equation}\label{e5.28}
y \sim z \Longleftrightarrow u(z-y) = 0.
\end{equation}
\end{theorem}

\begin{proof}
It was shown in Section 4 (see the paragraph below Proposition \ref{pr4.11}) that each compatible weak preference $\prec$ can be associated with the family of linear functions ${\mathcal{F}}_\prec := \{\phi_E \mid E \in {\mathcal O}(P_\prec), \phi_E \,\,\text{satisfies}\,\,\eqref{e4.18}\}$. It follows from Theorem \ref{th4.18} that ${\mathcal{F}}_\prec$ ordered by the linear order $\unlhd\,'$ is a cortege of linear functions on $Y$. Using the equivalence $y \prec z \Longleftrightarrow z-y \in P_\prec = \bigsqcup\{E \mid E \in {\mathcal O}(P_\prec)\}$ ($\bigsqcup$ denotes a disjoint union) and the characterization of $E \in {\mathcal O}(P_\prec)$ given in Proposition \ref{pr4.17} it is not difficult to verify that for the step-linear function $u_\prec : Y \to {\mathbb{R}}$ generated by the cortege ${\mathcal{F}}_\prec$ we have $y \prec z \Longleftrightarrow u_\prec(z-y) > 0$ and $y \sim z \Longleftrightarrow u_\prec(z-y) = 0$.

Conversely, let $u:Y \to {\mathbb{R}}$ be a step-linear function and let a binary relation $\prec \,\subset Y\times Y$ be defined by \eqref{e5.27}. It follows directly from the definition of $\prec$ and from the homogeneity of $u$ that the relation $\prec$ is compatible with algebraic operations of $Y$. Furthermore, the homogeneity of $u$ implies that $\prec$ is asymmetric, while the transitivity of $\prec$ is the corollary of the implication \eqref{e5.23}. Thus, the relation $\prec$ is a compatible partial preference. The indifference relation $\asymp$ generated by $\prec$ is defined as follows: $y \asymp z \Longleftrightarrow u(y-z) =0$. From this, using the implication \eqref{e5.22}, it is easily deduced  that $\asymp$ is transitive and hence $\prec$ is actually a compatible weak preference. Furthermore, since $\asymp = \sim$ for compatible weak preferences we conclude that \eqref{e5.28} also is true.
\end{proof}

In the next theorem we describe the class of such compatible weak preferences for which in \eqref{e5.27} and \eqref{e5.28} step-linear functions can be replaced by linear functions.

\begin{theorem}\label{th5.8}
Let $\prec$ be a compatible weak preference defined on a real vector space $Y$. Then the following statements are equivalent:
\begin{itemize}
\item[(i)]
the positive cone $P_\prec$ is algebraic open, i.e., ${\rm cr}P_\prec \ne \varnothing$ and $P_\prec = {\rm cr}P_\prec$;

\item[(ii)]
the family of open components of $P_\prec$ is a singleton, that is ${\mathcal O}(P_\prec) = \{P_\prec\}$;

\item[(iii)]
there exists a linear function $\phi:Y \to {\mathbb{R}}$ such that
\begin{equation}\label{e5.29}
y \prec z \Longleftrightarrow \phi(y) < \phi(z)\,\,\text{and}\,\,y \sim z \Longleftrightarrow \phi(y) = \phi(z).
\end{equation}
\end{itemize}
\end{theorem}

\begin{proof}
$(i) \Longleftrightarrow (ii)$. Let the statement $(i)$ be satisfied. Since ${\rm cr}P_\prec \ne \varnothing$, then ${\rm icr}P_\prec = {\rm cr}P_\prec$ and, hence, by Theorem \ref{th3.2} we have $P_\pprec = P_\prec$. Recall, that $P_\pprec := \{y \in P_\prec \mid z \unlhd y\,\,\forall\,\,z \in P_\prec\}$ is the set of strongly positive vectors and, consequently, through the equality $P_\pprec = P_\prec$ we obtain that $y \eq z$ for all $y,z \in P_\prec$. It proves that ${\mathcal O}(P_\prec) = \{P_\prec\}$.

Conversely, let ${\mathcal O}(P_\prec) = \{P_\prec\}$. Since each open component of $P_\prec$ is relatively algebraic open then $P_\prec = {\rm icr}P_\prec$. But, $\prec$ is a compatible weak preference, therefore $Y = P_\prec \bigcup L_\prec \bigcup (-P_\prec)$ and, consequently, ${\rm Lin}(P_\prec) =Y$. It implies that ${\rm icr}P_\prec = {\rm cr}P_\prec = P_\prec$. Thus, $P_\prec$ is algebraic open.

The proof of the equivalence $(i) \Longleftrightarrow (ii)$ is complete.

$(i) \Longleftrightarrow (iii)$. It follows from $(i)$ and the equality $Y = P_\prec \bigcup L_\prec \bigcup (-P_\prec)$ that $P_\prec$ is algebraic open halfspace in $Y$ and  $L_\prec$ is a homogeneous hyperplane which is the boundary of $P_\prec$. From this we conclude that there exists a nonzero linear function $\phi: Y \to {\mathbb{R}}$ such that $L_\prec =\{y \in Y \mid \phi(y)=0\}$ and $P_\prec = \{y \in Y \mid \phi(y)>0\}$. It follows from \eqref{e3} and \eqref{e3a} that the function $\phi$ satisfies \eqref{e5.29}.
This completes the proof of the implication $(i) \Longrightarrow (iii)$.

It is easily verified that the converse implication $(iii) \Longrightarrow (i)$ also is true.
\end{proof}

\begin{example}
Let a binary relation $\prec$ be defined on ${\mathbb{R}}^3$ as follows: $(y^1,y^2,y^3) \prec (z^1,z^2,z^3) \Longleftrightarrow y^1<z^1\,\,\text{or}\,\,y^1=z^1,y^2<z^2$. It is easy to verify that $\prec$ is a compatible weak preference with the equipotency relation corresponding $\prec$ being defined by $(y^1,y^2,y^3) \prec (z^1,z^2,z^3) \Longleftrightarrow y^1=z^1, y^2=z^2$. For this compatible weak preference $\prec$ the step-linear function $u: {\mathbb{R}}^3 \to {\mathbb{R}}$ such that $u(y^1,y^2,y^3) = y^1$ when $y^1 \ne 0$ and $u(y^1,y^2,y^3) = y^2$ when $y^2 = 0$
satisfies \eqref{e5.27} and \eqref{e5.28}, while there is no linear function which satisfies \eqref{e5.29}. Note, that ${\rm cr}P_\prec = \{(y^1,y^2,y^3) \in {\mathbb{R}}^3 \mid y^1 > 0\}$, while $P_\prec = \{(y^1,y^2,y^3) \in {\mathbb{R}}^3 \mid y^1 > 0\,\,\text{or}\,\,y^1=0,y^2>0\} \ne {\rm cr}P_\prec$ and, consequently, the statement $(i)$ from Theorem \ref{th5.8} does not hold.
\end{example}

\section{Extension of compatible partial preferences by compatible weak preferences}

We say that a partial preference $\prec_1$ \textit{extends} a partial preference $\prec_2$ if $\prec_2\,\subseteq\,\prec_1$ and we say that a partial preference $\prec_1$ \textit{regularly extends} a partial preference $\prec_2$ if $\prec_2\,\subseteq\,\prec_1$ and, in addition, $\sim_2\,\subseteq\,\sim_1$. Clearly, that a compatible partial preference $\prec_1$ (regularly) extends a compatible partial preference $\prec_2$ if and only if $P_{\prec_2}\,\subseteq\,P_{\prec_1}$ (and $L_{\prec_2}\,\subseteq\, L_{\prec_1})$.

\begin{example}
Consider two partial preferences $\prec_1$ and $\prec_2$ defined on ${\mathbb R}^3$ as follows: $(y_1,y_2,y_3) \prec_1  (z_1,z_2,z_3) \Longleftrightarrow (y_1 < z_1, y_2 < z_2)\vee (y_1 = z_1, y_2 = z_2, y_3 < z_3)$ and $(y_1,y_2,y_3) \prec_2  (z_1,z_2,z_3) \Longleftrightarrow y_1 < z_1, y_2 < z_2$. It is easy to see that $\prec_1$ extends $\prec_2$ but the extension $\prec_2$ to $\prec_1$ is not regular.
\end{example}

To prove the existence of regular compatible extensions for each compatible partial preference we need the following theorem on separation of convex cones by conical halfspaces \cite{GS98a,Gor21} which is presented here without proof.

\begin{theorem}\label{th6.1}
Let  $K_{1}$ and $K_{2}$ be convex cones in a real vector space $Y$, and let $K_{1}$ be asymmetric. Then
$K_{1}\cap K_{2}=\varnothing$ if and only if there exists an asymmetric conical halfspace $H \subset Y$ such that
$K_{1}\subset H$ and $K_2 \subset Y\setminus H$ or,
 equivalently, if and only if there exists a step-linear function
$u : Y \rightarrow {\mathbb R}$ such that
$$
u\,(x) > 0\,\,\text{for all}\,\,y\in K_{1},\quad
u\,(x)\leq 0\,\,\text{for all}\,\,y\in K_{2}.
$$
\end{theorem}

Note that this theorem is the version of the Kakutani-Tukey theorem on separation of convex sets by halfspaces \cite{Kakutani,Tukey,HF}.

\begin{theorem}\label{th6.2}
Each compatible partial preference $\prec\,\subset Y\times Y$ can be regularly extended to a compatible weak preference $\prec'\, \subset Y \times Y$.
\end{theorem}

\begin{proof}
Consider a compatible partial preference $\prec$ defined on $Y$. Let $P_\prec$ be the positive cone of $\prec$ and let $L_\prec$ be the vector subspace associated with $P_\prec$. Since $P_\prec\bigcap L_\prec = \varnothing$, then, by Theorem \ref{th6.1} there exists an asymmetric conical halfspace $H \subset Y$ such that $P_\prec \subset H$ and $L_\prec \subset Y \setminus H$. It is easy to see that a compatible weak preference $\prec'$    such that $P_{\prec'} = H$ regularly extends $\prec$.
\end{proof}

Theorem \ref{th6.2} is related to the Szpilrajn theorem \cite{Szpilrajn} on the extension of a partial order to a linear order.

\begin{theorem}\label{th6.3}
For any compatible partial preference $\prec$ defined on a real vector space $Y$ the family ${\mathcal U}_\prec$ of all step-linear function $u: Y \to {\mathbb{R}}$ such that
\begin{equation}\label{e6.1}
y \prec z  \Longrightarrow u(z-y)>0,
\end{equation}
and
\begin{equation}\label{e6.2}
y \sim z \Longrightarrow u(z-y)= 0.
\end{equation}
is nonempty and, moreover,
\begin{equation}\label{e6.3}
y \prec z  \Longleftrightarrow u(z-y)>0\,\,\text{for all}\,\,u \in {\mathcal U}_\prec
\end{equation}
and
\begin{equation}\label{e6.4}
y \sim z \Longleftrightarrow u(z-y)= 0\,\,\text{for all}\,\,u \in {\mathcal U}_\prec.
\end{equation}
\end{theorem}

\begin{proof}
By Theorem \ref{th6.2} any compatible partial preference $\prec$ can be regularly extended to a compatible weak preference $\prec'$. Due to Theorem \ref{th5.7} for $\prec'$ there exists a step-linear function $u':Y \to {\mathbb R}$ satisfying \eqref{e5.27} and \eqref{e5.28}. Since $\prec\,\subset \prec'$ and $\sim\,\subset \sim'$, the step-linear function $u'$ also satisfies \eqref{e6.1} and \eqref{e6.2} and, hence, the family ${\mathcal U}_\prec$ is nonempty.

Let us now prove \eqref{e6.3} and \eqref{e6.4}. The right implications ($\Longrightarrow$) in \eqref{e6.3} and \eqref{e6.4} follows from the definition of the family ${\mathcal U}_\prec$. To prove the left implication ($\Longleftarrow$) in \eqref{e6.3} it is sufficient to show that for any $\bar{y}, \bar{z} \in Y$ such that $\bar{y}\hspace{5pt}/\hspace{-11pt}\prec \bar{z}$ there exists a step-linear function $\bar{u} \in {\mathcal U}_\prec$ for which the equality $\bar{u}(\bar{z} - \bar{y}) \leq 0$ holds. If $\bar{y}\hspace{5pt}/\hspace{-11pt}\prec \bar{z}$ then either $\bar{z} \prec \bar{y}$ or $\bar{z} \asymp \bar{y}$. In the case when $\bar{z} \prec \bar{y}$ we have through \eqref{e6.1} that $u(\bar{z} - \bar{y}) = -u(\bar{y} - \bar{z})<0$ for all $u \in {\mathcal U}_\prec$. Consider the case when $\bar{z} \asymp \bar{y}$. Then, either $\bar{z} \sim \bar{y}$ or $(\bar{z} \asymp \bar{y})\wedge (\bar{z}\hspace{5pt}/\hspace{-10pt}\sim \bar{y})$ (here $\wedge$ is the logical symbol `and'). If $\bar{z} \sim \bar{y}$ through \eqref{e6.2} we have $u(\bar{z} - \bar{y}) = 0$ for all $u \in {\mathcal U}_\prec$. Thus, when $\bar{z} \prec \bar{y}$ or $\bar{z} \sim \bar{y}$ we can take as $\bar{u}$ any step-linear function $u \in {\mathcal U}_\prec$. It remains to consider the case $(\bar{z} \asymp \bar{y})\wedge (\bar{z}\hspace{5pt}/\hspace{-10pt}\sim \bar{y})$, which is equivalent to the condition that $p:=\bar{z}-\bar{y} \in Y \setminus (P_\prec \bigcup L_\prec \bigcup (-P_\prec))$. Note, that $p \ne 0$ since $p \not\in L_\prec$. Let $l^+_p:=\{\alpha p \mid \alpha > 0\}$ be the ray emanating from the origin and going through $p$. Prove that ${\rm conv}(P_\prec\bigcup (-l^+_p)) \bigcap L_\prec = \varnothing$, where ${\rm conv}M$ is the convex hull of the set $M$. Assume on the contrary that there exist $y \in P, \alpha > 0$ and $\lambda \in (0,1)$ such that $\lambda y - (1 - \lambda)\alpha p \in L_\prec$. Since $L_\prec$ is a vector space, $-y + \lambda^{-1}(1 - \lambda)\alpha p \in L_\prec$. By the definition of $L_\prec$ we obtain $y + (-y + \lambda^{-1}(1 - \lambda)\alpha p) = \lambda^{-1}(1 - \lambda)\alpha p  \in P_\prec$, which contradicts the condition $p \not\in P_\prec$. The obtained contradiction proves the condition ${\rm conv}(P_\prec\bigcup (-l^+_p)) \bigcap L_\prec = \varnothing$. It is easy to verify that ${\rm conv}(P_\prec\bigcup (-l^+_p))$ is an asymmetric convex cone. Consequently, by Theorem \ref{th6.1} there exists a step-linear function $\bar{u}:Y \to {\mathbb R}$ such that $\bar{u}(y) > 0$ for all $y \in {\rm conv}(P_\prec\bigcup (-l^+_p))$ and $\bar{u}(y) \leq 0$ for all $L_\prec$. Since the function $\bar{u}$ is homogeneous and $-p \in {\rm conv}(P_\prec\bigcup (-l^+_p))$ we have that $\bar{u}(p) < 0$ and, since $L_\prec$ is a vector space, we have $\bar{u}(y)=0$ for all $y \in L_\prec$. This completes the proof of \eqref{e6.3}.

To prove the left implication ($\Longleftarrow$) in \eqref{e6.4} we show that for any $y,z \in Y$ such that $y \not\sim z$ there exists $u \in {\mathcal U}_\prec$ for which $u(z-y) \ne 0$. If $y \not\sim z$ then the following three disjoint alternatives can be held: $y\prec z,\,z \prec y$ and $(z \asymp y)\wedge (z \hspace{5pt}/\hspace{-10pt}\sim y)$. For $y\prec z,\,z \prec y$ through \eqref{e6.1} we have $u(z-y) \ne 0$ for all $u \in {\mathcal U}_\prec$. While for the alternative $(z \asymp y)\wedge (z \hspace{5pt}/\hspace{-10pt}\sim y)$ it was shown above there exists a step-linear function $u \in {\mathcal U}_\prec$ such that $u(z-y) < 0$. Thus, if $u(z-y)= 0\,\,\text{for all}\,\,u \in {\mathcal U}_\prec$ then $y \sim z$.

This completes the proof of the theorem.
\end{proof}

Theorem \ref{th6.3} can be considered as the analytical version of the Dushnik-Miller theorem \cite{Dushnik} adapted to compatible partial preference.

\begin{theorem}
If a compatible partial preference $\prec$ defined on a real vector space $Y$ is relatively algebraic open then there exists a linear function $l:Y \to {\mathbb{R}}$ such that
\begin{equation}\label{e6.5}
y \prec z  \Longrightarrow \phi(y) < \phi(z)\,\,\text{and}\,\,y \sim z \Longrightarrow \phi(y) = \phi(z).
\end{equation}
\end{theorem}

\begin{proof}
Since $P_\prec$ is relatively algebraic open  and $P_\prec\bigcap L_\prec = \varnothing$, then due to \cite[Corollary 5.62]{Aliprantis} there exists a linear function $\phi: Y \to {\mathbb{R}}$ such that $\phi(y) > 0$ for all $y \in P_\prec$ and $\phi(z) = 0$ for all $z \in L_\prec$. It follows from \eqref{e3} and \eqref{e3a} that the linear function $\phi$ satisfies \eqref{e6.5}.
\end{proof}


\section*{Disclosure statement}

There are no conflicts of interest to disclose.

\section*{Funding}

The research was supported by the State Program for Fundumental Research of Republic of Belarus 'Convergence-2025'.

\end{document}